\newcommand{\isarxiv}{}

\ifdefined\isarxiv
    \documentclass[letterpaper,11pt]{article}
    \usepackage{fullpage}
\else
    \documentclass[12pt,letterpaper,reqno]{article}
    \usepackage{fullpage,etoolbox}
    \usepackage{mlmodern} % font

    \makeatletter
    \patchcmd{\@maketitle}{\normalsize}{\large}{}{}
\fi

\usepackage[T1]{fontenc}
\usepackage[utf8]{inputenc}
\usepackage[english]{babel}
\usepackage{amsmath,amsfonts,amssymb,amsthm,mathtools,dsfont,bbm}
\usepackage{graphicx,adjustbox,color,xcolor}
\usepackage{tikz,pgfplots,subcaption}
\usepackage{algorithm,algpseudocode}
\usepackage{booktabs,threeparttable,array,pifont} % table related
\usepackage{enumitem}
\usepackage{cite}
\usepackage{hyperref}
\usepackage[capitalize,nameinlink,noabbrev]{cleveref}
\usepackage{cancel} % strikethrough

\allowdisplaybreaks[4]
\mathtoolsset{centercolon}
\pgfplotsset{compat=1.18}
\DeclareMathAlphabet{\mymathbb}{U}{BOONDOX-ds}{m}{n}
\hypersetup{colorlinks, hypertexnames=false, pageanchor=true,
            linkcolor=blue, citecolor={green!50!black}, urlcolor=cyan
            }

\numberwithin{equation}{section}

\newtheorem{theorem}{{Theorem}}[section]

\newtheorem{corollary}[theorem]{{Corollary}}
\newtheorem{proposition}[theorem]{{Proposition}}

\crefname{assumption}{Assumption}{Assumptions}

\newcommand{\reals}                  {\mathbb R}

\newcommand{\natint}                 {\mathbb N}

\newcommand{\argmin}                 {\mathop{\mathrm{argmin}}}

\newcommand{\minimize}               {\text{\upshape minimize}}
\newcommand{\maximize}               {\text{\upshape maximize}}

\newcommand{\subg}                   {\widetilde \nabla}

\DeclareMathOperator{\dom}           {{\bf dom}}                  % domain
\DeclareMathOperator{\prox}          {{\bf prox}}                 % proximal mapping
\definecolor{uclablue}{RGB}{39,116,174}
\definecolor{ogreen}{RGB}{60,128,49}

\definecolor{myblue}{RGB}{0,114,178}
\definecolor{mygreen}{RGB}{0,158,115}
\definecolor{myorange}{RGB}{213,94,0}

\newcommand{\eg}{{\it e.g.}}
\newcommand{\ie}{{\it i.e.}}

\newcommand{\bbN}{\mathbb{N}}

\newcommand{\cL}{\mathcal{L}}

\newcommand{\cO}{\mathcal{O}}

\newcommand{\cS}{\mathcal{S}}

\newcommand{\oline}[1]{\mkern 1.5mu\overline{\mkern-1.5mu#1}}

\renewcommand{\hbar}{\oline{h}}

\newcommand{\pbar}{\oline{p}}

\newcommand{\ubar}{\oline{u}}

\newcommand{\wbar}{\oline{w}}
\newcommand{\xbar}{\oline{x}}

\newcommand{\zbar}{\oline{z}}

\newcommand{\utilde}{\tilde{u}}

\newcommand{\xtilde}{\tilde{x}}

\makeatletter
\renewcommand*\env@matrix[1][c]{\hskip -\arraycolsep
  \let\@ifnextchar\new@ifnextchar
  \array{*\c@MaxMatrixCols #1}}
\makeatother

\usepackage{xparse}
\DeclareFontFamily{U}{ntxmia}{}
\DeclareFontShape{U}{ntxmia}{m}{it}{<-> ntxmia }{}
\DeclareFontShape{U}{ntxmia}{b}{it}{<-> ntxbmia }{}
\DeclareSymbolFont{lettersA}{U}{ntxmia}{m}{it}
\SetSymbolFont{lettersA}{bold}{U}{ntxmia}{b}{it}

\ExplSyntaxOn
\NewDocumentCommand{\varmathbb}{m}
 {
  \tl_map_inline:nn { #1 }
  {
    \use:c { varbb##1 }
  }
 }
\tl_map_inline:nn { ABCDEFGHIJKLMNOPQRSTUVWXYZ }
 {
  \exp_args:Nc \DeclareMathSymbol{varbb#1}{\mathord}{lettersA}{\int_eval:n { `#1+67 }}
 }
\exp_args:Nc \DeclareMathSymbol{varbbk}{\mathord}{lettersA}{169}
\ExplSyntaxOff

\usepackage{makecell}

\newcommand{\ineqs}{I}
\newcommand{\tildeineqs}{\tilde{I}}
\newcommand{\sos}{\cS}
\newcommand{\tildesos}{\tilde{\cS}}
\newcommand{\lhs}{\text{LHS}}
\newcommand{\xdif}{v}
\newcommand{\CCP}{CCP}

\newcommand{\norm}[1]{\left\| #1 \right\|} 
\newcommand{\inner}[2]{ \left\langle #1 ,  #2 \right\rangle }
\newcommand{\set}[1]{\left\{  #1  \right\}}

\newcommand{\pr}[1]{ \left( #1 \right) }

\def\keywordsname{{\bfseries Keywords:}\enspace}
\def\keywords#1{\par\addvspace\medskipamount{\rightskip=0pt plus1cm
\def\and{\ifhmode\unskip\nobreak\fi\ $\cdot$
}\noindent\keywordsname\ignorespaces#1\par}}
\def\subclassname{{\bfseries Mathematics Subject Classification
(2010):}\enspace}
\def\subclass#1{\par\addvspace\medskipamount{\rightskip=0pt plus1cm
\def\and{\ifhmode\unskip\nobreak\fi\ $\cdot$
}\noindent\subclassname\ignorespaces#1\par}}

\title{Exact worst-case convergence rates for Douglas--Rachford and Davis--Yin splitting methods}

\author{Edward Duc Hien Nguyen%
\thanks{Department of Electrical and Computer Engineering, Rice University. Email: \texttt{en18@rice.edu}.}
\and
Jaewook J.\ Suh%
\thanks{Department of Computational Applied Mathematics and Operations Research, Rice University. Email: \texttt{\{jacksuh, shiqian.ma\}@rice.edu}.}
\and
Xin Jiang%
\thanks{Department of Industrial and Systems Engineering, University of Houston. Email: \texttt{xinjiang@uh.edu}.}
\and Shiqian Ma\footnotemark[2]}

\date{\today}

\begin{document}

\maketitle

\begin{abstract}
In this work, we aim to establish the exact worst-case convergence rates of Douglas--Rachford splitting (DRS) and Davis--Yin splitting (DYS) when applied to convex optimization problems. Both DRS and DYS have two variants as swapping the roles of the two nonsmooth convex functions in both algorithms yields different sequences of iterates. For both variants of DRS and one variant of DYS, we establish the exact worst-case convergence rates, including the constant factor, using the primal--dual gap function as the performance metric. We provide worst-case examples to verify the tightness of these rates. To the best of our knowledge, this is the first result that establishes the exact worst-case convergence rates for DRS and DYS that include the constant factor. For the other variant of DYS, we establish the best-known convergence rate and provide a concrete example indicating a discrepancy between the convergence rates of the two DYS variants.
\end{abstract}

\subclass{65K05 \and 68Q25 \and 90C25}

\ifdefined\isarxiv

\else
    \tableofcontents
    \newpage
\fi

%%%%%%%%%%%%%%%%%%%%%%%%%%%%%%%%%%%%%%%%%%%%%%%%%%%%%%%%%%%%%%%%%%%%%%%%%%%%%%%%%%%%%%%%%%%%%%%%%%%%%%%%%%%%%%
\section{Introduction}

We study proximal splitting methods for solving convex optimization problems in the form
\begin{equation} \label{eq:prob-primal}
    \begin{array}{cc}
        \underset{x \in \reals^n}{\minimize} & f(x) + g(x) + h(x),
    \end{array}
\end{equation}
where $f$, $g$, and $h$ are closed convex proper (CCP) functions and $h$ is differentiable. This general problem covers a wide variety of applications in machine learning, signal and image processing, operations research, control, and other fields \cite{CKCH23,CP11b,CP16a,PB14,RY22}. A well-known method for solving problem \eqref{eq:prob-primal} is the Davis--Yin splitting (DYS) algorithm \cite[Algorithm~1]{DY17b}:
\begin{equation} \label{eq:dys-fpi} \tag{DYS}
    \begin{split}
        w^{k+1} &= \prox_{\alpha g} (y^{k}) \\
        x^{k+1} &= \prox_{\alpha f} (2w^{k+1} - y^{k} - \alpha \nabla h(w^{k+1})) \\
        y^{k+1} &= y^{k} + x^{k+1} - w^{k+1}.
    \end{split}
\end{equation}
\eqref{eq:dys-fpi} recovers several classical proximal methods when parts of problem \eqref{eq:prob-primal} vanish. For example, it reduces to the Douglas--Rachford splitting (DRS) algorithm \cite{EB92,DR56,LM79,PR55} when $h=0$, and to the forward--backward splitting (FBS) algorithm \cite{LM79,Passty79} when either $f$ or $g$ vanishes.

{One may observe that $f$ and $g$ play symmetric roles in problem \eqref{eq:prob-primal}, and thus swapping them does not alter the problem at all. Yet this symmetry does not carry over to the algorithmic level. Swapping $f$ and~$g$ in DRS and \eqref{eq:dys-fpi} generally leads to a different sequence of variables and thus a different algorithm \cite{Eckstein89,YY17}. In view of this subtle yet critical distinction, we analyze the convergence rate of the four algorithms: DYS, DRS, and their swapped versions, with the primal--dual gap function as the performance measure. Specialized analysis is needed for the swapped variants because, for the primal--dual gap function, the analysis does not directly extend to the swapped algorithm. As a consequence of our tight analysis of these algorithms, we found, perhaps surprisingly, the swapped version of \eqref{eq:dys-fpi} achieves a faster worst-case ergodic rate of convergence. This highlights the tangible impact of update order on algorithmic performance.}

\paragraph{Contributions.}
The contributions of this paper are summarized as follows.
\begin{itemize}
    \item We provide novel convergence analyses of DYS, DRS, and their swapped versions, by deriving equalities that explicitly reveal the critical inequalities used in the proof. These equalities not only imply immediately the ergodic convergence rates for the primal--dual gap function but also provide guidance for constructing worst-case examples for each algorithm.

    \item {We establish the exact worst-case rate of convergence, including the constant factor, for both variants of DRS and one variant of DYS. In these three cases, the tightness of the established rates is confirmed through worst-case examples. To the best of our knowledge, this is the first result that establishes the exact worst rate of convergence for DRS and DYS that includes the constant factor. For the other variant of DYS, we establish the best-known convergence rate and verify the discrepancy between the two DYS variants via a concrete example.} 

    \item Our analyses reveal that the swapped version of DYS converges faster than the original, whereas both variants of DRS share the same worst-case rate. This appears to be the first result that formally distinguishes between the two variants of proximal algorithms based on their convergence rates.
\end{itemize}

\paragraph{Outline.}
The rest of the paper is organized as follows. \Cref{sec:bg} reviews fundamental concepts from convex analysis and summarizes existing convergence results for DYS and DRS. In \cref{sec:drs}, we analyze the convergence of two DRS variants and establish the tightness of the results via worst-case examples. Similarly, \cref{sec:dys} presents analyses of two DYS variants, demonstrating a difference in their convergence rates via concrete examples. Finally, \cref{sec:conclusion} concludes the paper.

%%%%%%%%%%%%%%%%%%%%%%%%%%%%%%%%%%%%%%%%%%%%%%%%%%%%%%%%%%%%%%%%%%%%%%%%%%%%%%%%%%%%%%%%%%%%%%%%%%%%%%%%%%%%%%
\section{Background material and prior work} \label{sec:bg}

In this section, we review fundamental concepts from convex optimization and introduce the notation used throughout the paper. We also present two variants of DRS and DYS, along with their known convergence results. Although both DRS and DYS were originally proposed to solve monotone inclusion problems, we focus in this paper on the seemingly more restrictive setting of convex optimization. This choice is motivated by the fact that the discrepancy in convergence rates between the two variants of DYS arises specifically within this narrower setting. In particular, the illustrative examples presented later involve only convex functions.

Throughout the paper, we use the notation $\inner{x}{y} = x^Ty$ for the standard inner product of vectors $x$ and~$y$, and $\|x\| = \inner{x}{x}^{1/2}$ for the Euclidean norm of a vector $x$. 
Also, we denote $\natint$ (resp.,~$\natint_+$) as the set of nonnegative (resp., positive) integers, \ie, $\natint \coloneqq \set{0,1,2,\dots}$ and $\natint_+ \coloneqq \set{1,2,\dots}$.

\subsection{Basic concepts and notation in convex optimization}

We follow the standard definitions in convex optimization; see, \eg, \cite{BV04,BC17,Nesterov18,Rockafellar70,RY22}. We denote the subdifferential of a convex function $f \colon \reals^n \to \reals$ as $\partial f$, defined by $\partial f(x) := \{ v \in \reals^n \mid f(y) \ge f(x) + \inner{v}{y-x} \ \text{for all} \ y \in \reals^n\}$. An element of $\partial f(x)$ is called a subgradient of $f$ at $x$, and when its choice is unambiguous, we use the shorthand $\subg f(x)$, following the notation introduced in \cite{Bertsekas11}: $\subg f(x) \in \partial f(x)$. With this, the inequality in the definition of $\partial f$ becomes
\begin{equation} \label{eq:cvx}
    f(x) - f(y) + \langle{\subg f(x)}, {y-x} \rangle \le 0 \quad \text{for all} \ y \in \dom f.
\end{equation}
The notation $\subg f$ is particularly useful in studying the proximal operator of a CCP function~$f$:
\[
    \prox_{f}(x) := \argmin_{y \in \reals^n} \{f(y) + \tfrac{1}{2} \|y - x\|^2\}.
\]
We often denote $\subg f(\prox_{\alpha f} (x)) := \tfrac{1}{\alpha}(\prox_{\alpha f} (x) - x) \in \partial f(\prox_{\alpha f} (x))$, and then the proximal operator of~$\alpha f$ (with $\alpha > 0$) can be written as:
\begin{equation} \label{eq:prox-op}
    \prox_{\alpha f} (x) = x - \alpha \subg f (\prox_{\alpha f} (x)).
\end{equation}
Since $\prox_{\alpha f}$ is well defined when $f$ is CCP, the quantity $\subg f(\prox_{\alpha f}(x))$ is uniquely determined. 

For any function $f$, its Fenchel conjugate is defined as $f^\ast(y) := \sup_x \set{\inner{y}{x} - f(x) }$. When $f$ is CCP, it holds that
\begin{equation} \label{eq:fenchel}
    f^\ast(y) = \inner{y}{x} - f(x) \qquad \Longleftrightarrow \qquad y \in \partial f(x)
    \qquad \Longleftrightarrow \qquad x \in \partial f^\ast (y),
\end{equation}
which is called Fenchel's identity, and the equality
\begin{equation} \label{eq:moreau}
    \prox_{\alpha f} (x) + \alpha \prox_{\alpha^{-1} f^\ast} (x/\alpha) = x
\end{equation}
holds for all $x$ and all $\alpha > 0$, which is known as the Moreau identity.

We say $f$ is $L$-smooth if it is differentiable and its gradient is $L$-Lipschitz continuous. When $f$ is $L$-smooth and convex, it satisfies the inequality
\begin{equation} \label{eq:smth}
    f(x) - f(y) + \inner{\nabla f(x)}{y-x} + \frac{1}{2L} \norm{\nabla f(x) - \nabla f(y)}^2 \le 0 \quad \text{for all} \ x,y \in \dom f.
\end{equation}

\paragraph{Dual problem and optimality conditions.}
The dual of problem~\eqref{eq:prob-primal} is
\begin{equation} \label{eq:prob-dual}
    \begin{array}{ll}
        \underset{u \in \reals^n}{\maximize} & -(f + h)^\ast (-u) - g^\ast(u),
    \end{array}
\end{equation}
where the conjugate $(f+h)^\ast$ is the infimal convolution of $f^\ast$ and $h^\ast$:
\[
    (f+h)^\ast (u) = \inf_w \set{f^\ast(w) + h^\ast(u-w)}.
\]
The primal--dual optimality conditions for \eqref{eq:prob-primal} and \eqref{eq:prob-dual} are
\begin{equation} \label{eq:opt-cond}
0 \in \partial f(x) + \nabla h(x) + u, \qquad 0 \in \partial g^\ast (u) - x.
\end{equation}
Throughout the paper, we assume \eqref{eq:opt-cond} is solvable, so strong duality holds between \eqref{eq:prob-primal} and \eqref{eq:prob-dual}.

We will refer to the convex--concave function
\begin{equation} \label{eq:lagrangian}
    \cL (x, u) = f(x) + h(x) + \langle u, x \rangle - g^\ast (u)
\end{equation}
as the \textit{Lagrangian} of \eqref{eq:prob-primal}. We follow the convention that $\cL(x,u) = +\infty$ if $x \notin \dom (f + h)$ and $\cL(x,u) = -\infty$ if $x \in \dom (f+h)$ and $z \notin \dom g^\ast$. The objectives in \eqref{eq:prob-primal} and the dual problem~\eqref{eq:prob-dual} can be expressed as
\[
\sup_u \cL(x,u) = f(x) + g(x) + h(x), \qquad \inf_x \cL(x,u) = -(f+h)^\ast (-u) - g^\ast(u).
\]
Solutions $x^\star$, $u^\star$ of the optimality conditions \eqref{eq:opt-cond} form a saddle point of $\cL$:
\[
\inf_x \sup_u \cL(x,u) = \sup_u \cL(x^\star, u) = \inf_x \cL(x, u^\star) = \sup_u \inf_x \cL(x,u).
\]
Then, it holds that
\[
\cL (x^\star, u) \leq \cL (x^\star, u^\star) \leq \cL (x, u^\star)
\]
for all $x \in \dom f$ and $u \in \dom g^\ast$. In particular, $\cL(x^\star, u^\star)$ is the optimal value of \eqref{eq:prob-primal} and \eqref{eq:prob-dual}, and the pair of primal--dual problems is equivalent to the saddle point problem
\begin{equation} \label{eq:prob-saddle-pt}
    \begin{array}{lll}
        \underset{x \in \reals^n}{\minimize} & \underset{u \in \reals^n}{\maximize} & \cL (x, u).
    \end{array}
\end{equation}

\subsection{Douglas--Rachford splitting algorithms}

In this section, we present two variants of DRS for solving the problem \eqref{eq:prob-primal} in the special case where $h=0$. Although these algorithms can be viewed as special cases of DYS, we present them explicitly to highlight that the discrepancy in convergence rates between the two variants of DYS arises from the presence of the smooth function $h$.

\paragraph{DRS-$gf$.}
The Douglas--Rachford splitting algorithm \cite{CP07,DR56,EB92,LM79,PR55} for solving \eqref{eq:prob-primal} with $h=0$ reads as
\begin{subequations} \label{eq:drs-fpi}
    \begin{align}
        w^{k+1} &= \prox_{\alpha g} (y^{k}) \\
        x^{k+1} &= \prox_{\alpha f} (2w^{k+1} - y^{k}) \\
        y^{k+1} &= y^{k} + x^{k+1} - w^{k+1},
    \end{align}
\end{subequations}
where $\alpha > 0$ is an algorithm parameter. Then, we define $u^{k+1} = \prox_{\alpha^{-1} g^\ast} (\tfrac{1}{\alpha}y^{k})$ and eliminate $w^{k+1}$ and $y^{k+1}$ to obtain
\begin{equation} \label{eq:drs-gf} \tag{DRS-$gf$}
    \begin{split}
        u^{k+1} &= \prox_{\alpha^{-1} g^\ast} (u^{k} + \tfrac{1}{\alpha} x^{k}) \\
        x^{k+1} &= \prox_{\alpha f} (x^{k} - \alpha (2 u^{k+1} - u^{k})),
    \end{split}  
\end{equation}
which we will refer to as DRS-$gf$, as it calls $g$ first and then $f$. The presented form \eqref{eq:drs-gf} solves the saddle point problem \eqref{eq:prob-saddle-pt} with $h=0$ and has been studied in, \eg, \cite[Eq.~(5.18)]{CKCH23} and \cite[\S3]{JV23}. Its convergence has been analyzed under various regularity conditions and here we focus on the most basic setting where both $f$ and $g$ are CCP functions. In this case, convergence of the iterates follows readily, as \eqref{eq:drs-fpi} can be interpreted as an instance of the proximal point method \cite{EB92}. Specifically, the iterates $(x^k, w^k, y^k)$ generated by \eqref{eq:drs-fpi} converge to $(x^\star, x^\star, x^\star + \alpha u^\star)$ and the iterates $(x^k,u^k)$ generated by \eqref{eq:drs-gf} converge to $(x^\star, u^\star)$. In addition, \eqref{eq:drs-gf} converges at a sublinear rate. More specifically, the non-ergodic sequence in \eqref{eq:drs-fpi} exhibits the following rate \cite[Theorem 4]{DY17c}, \cite{HY15}:
\begin{equation} \label{eq:drs-conv-o}
|f(x^K) + g(w^K) - f(x^\star) - g(x^\star)| = o(1/\sqrt{K+1}),
\end{equation}
while the ergodic sequence $(\xbar^K, \wbar^K)$ converges at a faster $\cO(1/(K+1))$ rate \cite[Theorem 3]{DY17c}:
\begin{equation} \label{eq:drs-conv-O}
|f(\xbar^K) + g(\wbar^K) - f(x^\star) - g(x^\star)| = \cO(1/(K+1)),
\end{equation}
where given $K \in \natint_+$, we define $\zbar^K \coloneqq \frac{1}{K} \sum_{k=1}^{K} z^k$ for $z \in \{x,w,u\}$. These rates are shown to be tight up to a constant \cite{DY17c}. Moreover, the $\cO(1/(K+1))$ ergodic rate remains valid when the primal--dual gap function is used \cite{BHG24,CP16b,JV23}.

\paragraph{DRS-$fg$.}
Since \eqref{eq:drs-fpi} is not symmetric in $f$ and $g$, exchanging $f$ and $g$ yields a different algorithm 
\begin{equation*}
\begin{split}
    x^{k+1} &= \prox_{\alpha f} (y^{k}) \\
    w^{k+1} &= \prox_{\alpha g} (2x^{k+1} - y^{k}) \\
    y^{k+1} &= y^{k} + w^{k+1} - x^{k+1}.
\end{split}
\end{equation*}
Letting $u^{k+1} = \alpha^{-1} (2x^{k+1} - y^k - w^{k+1})$ and applying Moreau identity \eqref{eq:moreau}, we can eliminate $w^{k+1}$ and $y^{k+1}$ and obtain an equivalent algorithm
\begin{equation} \label{eq:drs-fg} \tag{DRS-$fg$}
    \begin{split}
        x^{k+1} &= \prox_{\alpha f} (x^{k} - \alpha u^{k}) \\
        u^{k+1} &= \prox_{\alpha^{-1} g^\ast} (u^{k} + \tfrac{1}{\alpha}(2x^{k+1} - x^{k})),
    \end{split}
\end{equation}
which we will refer to as DRS-$fg$.

In general, \eqref{eq:drs-gf} and \eqref{eq:drs-fg} generate different iterates \cite{YY17} and are thus considered \textit{not equivalent}. Nevertheless, the convergence of \eqref{eq:drs-fg} iterates can be established via similar arguments, as \eqref{eq:drs-fg} can also be interpreted as an instance of the proximal point method with a different splitting strategy \cite{Eckstein89,EB92}. Moreover, since the objective gap is invariant in the order of $f$ and $g$, the convergence rate results \eqref{eq:drs-conv-o} and \eqref{eq:drs-conv-O} carry over directly to \eqref{eq:drs-fg}. In contrast, the primal--dual gap function treats $f$ and $g$ asymmetrically, so the convergence guarantees for \eqref{eq:drs-gf} established in \cite{BHG24,CP16b} do not automatically extend to \eqref{eq:drs-fg}. As in this paper we adopt the primal--dual gap function as the performance measure (see \cref{sec:metric} for further discussion), our analysis for \eqref{eq:drs-fg} must proceed separately, even though both algorithms ultimately exhibit the same worst-case rate.

\subsection{Davis--Yin splitting algorithms}

\paragraph{DYS-$gf$.}
The iterations \eqref{eq:dys-fpi} were first presented in \cite{DY17b} and are now referred to as the Davis--Yin splitting algorithm in the literature. Again, we introduce $u^{k+1} = \prox_{\alpha^{-1} g^\ast} (\tfrac{1}{\alpha}y^{k})$ and apply Moreau identity \eqref{eq:moreau} to eliminate $w^{k+1}$ and $y^{k+1}$:
\begin{equation} \label{eq:dys-gf} \tag{DYS-$gf$}
    \begin{split}
        u^{k+1} &= \prox_{\alpha^{-1} g^\ast} (u^{k} + \tfrac{1}{\alpha}x^{k}) \\
        x^{k+1} &= \prox_{\alpha f} (x^{k} - \alpha (2u^{k+1} - u^{k}) - \alpha \nabla h(x^{k} + \alpha (u^{k} - u^{k+1}))),
    \end{split}
\end{equation}
which we will refer to as DYS-$gf$. This form solves the saddle point problem \eqref{eq:prob-saddle-pt} and when $h=0$, \eqref{eq:dys-gf} reduces to \eqref{eq:drs-gf}. DYS was originally introduced as a fixed-point iteration scheme for solving monotone inclusion problems, so its convergence rate is often measured by the fixed-point residual; see, \eg, \cite{DY17b}. As in our discussion of DRS, we focus on the convex optimization setting under minimal assumptions, where $f$ and~$g$ are merely CCP functions. While convergence of the DYS iterates is well established \cite[\S4.1]{DY17b}, the convergence rate of \eqref{eq:dys-gf} (under the general convex setting) has received less attention. In existing works \cite{SCMR22,PG18,YL24,ZHT19}, DYS is typically viewed as a primal--dual proximal method, with the gap function used as the performance measure, and an $\cO(1/K)$ ergodic convergence rate is typically established. Yet the tightness of these rates is not discussed. Finally, we note that numerous variants of DYS have been proposed, including stochastic DYS \cite{YVC16,YGS21,ZC18,ZHT19}, inexact DYS \cite{ZTC18}, adaptive DYS \cite{PG18}, and inertial DYS \cite{CTY19}. Some of these consider settings different from ours and hence not discussed in detail.

\paragraph{DYS-$fg$.}
As for \eqref{eq:drs-gf}, we switch the role of $f$ and $g$ in \eqref{eq:dys-fpi} and obtain a different algorithm
\begin{equation} \label{eq:dys-fg} \tag{DYS-$fg$}
    \begin{split}
        x^{k+1} &= \prox_{\alpha f} (x^{k} - \alpha (u^{k} + \nabla h(x^{k}))) \\
        u^{k+1} &= \prox_{\alpha^{-1} g^\ast} (u^{k} + \tfrac{1}{\alpha} (2x^{k+1} - x^{k} + \alpha \nabla h(x^{k}) - \alpha \nabla h(x^{k+1}))),
    \end{split}
\end{equation}
which we will refer to as DYS-$fg$. When $h=0$, \eqref{eq:dys-fg} reduces to \eqref{eq:drs-fg}. This form \eqref{eq:dys-fg} was first presented as a special case of the PD3O algorithm \cite{Yan18}; see also \cite[Eq.~(8.6)]{CKCH23} and \cite[\S3]{JV23}. As noted earlier, \eqref{eq:drs-gf} and \eqref{eq:drs-fg} are not equivalent; nor are \eqref{eq:dys-gf} and \eqref{eq:dys-fg}. Consequently, the results in \cite{DY17b,PG18,YL24,ZHT19} for \eqref{eq:dys-gf} do not directly extend to \eqref{eq:dys-fg}, as they adopt performance measures that treat $f$ and $g$ asymmetrically. An $\cO(1/K)$ ergodic rate for \eqref{eq:dys-fg} is established in \cite[Theorem~2]{Yan18} (using a slightly different gap function, which starts at iteration $0$ rather than $1$) and \cite[Theorem~4]{JV23}, though the tightness of these rates is not addressed.

\subsection{Discussion on performance measures} \label{sec:metric}

Unlike gradient-based methods, the choice of the performance measure in DRS and DYS is somewhat arbitrary and often dictated by the analysis techniques employed. It is difficult to determine which metric is most natural or informative, as each comes with its own advantages and limitations. For example, the objective gap $|f(x^K) + g(w^K) - f(x^\star) - g(x^\star)|$, used in the seminal work \cite{DY17c}, does not capture the fact that $x^K$ and $w^K$ in \eqref{eq:drs-fpi} are two different sequences, and therefore $f(x^K) + g(w^K)$ does not correspond to the objective value of an iterate. Moreover, the objective value $f(x^K) + g(x^K)$ can be $+\infty$ as $x^K$ may not be in the domain of $g$. So the use of $|f(x^K) + g(x^K) - f(x^\star) - g(x^\star)|$ is valid only with some additional conditions, for instance, $g$ is additionally locally Lipschitz \cite{DY17c}.

Our analysis focuses on the convergence of the algorithms when applied to solving the saddle point problem~\eqref{eq:prob-saddle-pt} and studies the primal--dual gap function
\begin{equation}    \label{eq:gap-func}
    \cL (\xbar^K, u) - \cL(x, \ubar^K) = f(\xbar^K) + h(\xbar^K) + \langle u, \xbar^K \rangle - g^\ast (u) - \big(f(x) + h(x) + \langle \ubar^K, x \rangle - g^\ast(\ubar^K) \big).
\end{equation}
More precisely, the performance measure we choose is
\begin{equation} \label{eq:gap}
\sup_{\substack{x \in \dom f \\ u \in \dom g^\ast}} \set{\cL (\xbar^K, u) - \cL(x, \ubar^K)},
\end{equation}
which has been used extensively in the analysis of primal--dual proximal methods; see, \eg, \cite{Condat13,CP11a,CP16b,JV23,SCMR22,Yan18}. One shall note that the quantity $\cL(\xbar^K, u^\star) - \cL(x^\star, \ubar^K)$ is not a valid performance measure: while this quantity is $0$ when $(\xbar^K, \ubar^K)$ is a saddle point, the converse is not necessarily true \cite{CP11a}. So the pointwise supremum over $(x,u)$ is necessary in~\eqref{eq:gap}.

%%%%%%%%%%%%%%%%%%%%%%%%%%%%%%%%%%%%%%%%%%%%%%%%%%%%%%%%%%%%%%%%%%%%%%%%%%%%%%%%%%%%%%%%%%%%%%%%%%%%%%%%%%%%%%
\section{Convergence analysis of two variants of DRS} \label{sec:drs}

In this section, we analyze the convergence of \eqref{eq:drs-gf} and \eqref{eq:drs-fg} using the primal--dual gap function \eqref{eq:gap} as the performance measure. To demonstrate the tightness of our results, we construct worst-case examples for which the two variants of DRS converge at \textit{exactly} the established rates. Although the worst-case examples differ only by a sign, the analyses of \eqref{eq:drs-gf} and \eqref{eq:drs-fg} must be carried out separately. This separation also facilitates a clear comparison with the analyses of \eqref{eq:dys-gf} and \eqref{eq:dys-fg} in \cref{sec:dys}, whose rates, perhaps surprisingly, do not coincide.

\subsection{\texorpdfstring{DRS-$gf$}{DRS-gf}: worst-case rate and its tightness} \label{sec:drs-gf}

We begin our analysis of \eqref{eq:drs-gf} with another equivalent reformulation
\begin{subequations} \label{eq:drs-gf-reform}
    \begin{align}
        u^{k+1} &= u^k + \tfrac{1}{\alpha} x^k - \tfrac{1}{\alpha} \subg g^\ast (u^{k+1}) \label{eq:drs-gf-reform-u} \\
        x^{k+1} &= x^{k} - \alpha (2 u^{k+1} - u^{k}) - \subg f (x^{k+1}), \label{eq:drs-gf-reform-x}
    \end{align}
\end{subequations}
where we used \eqref{eq:prox-op}. For later use, we also introduce the notation
\begin{equation} \label{eq:drs-gf-p}
    p^{k} := \subg g^\ast (u^{k}) =  \alpha u^{k-1} + x^{k-1} - \alpha u^{k},
\end{equation}
which implies from \eqref{eq:fenchel} that $u^{k+1} \in \partial g(p^{k+1})$. We will see that how the use of $\subg$ in reformulating \eqref{eq:drs-gf}, inspired by \cite{Bertsekas11,DY17b}, facilitates the analysis of \eqref{eq:drs-gf}.

Similarly, for the ergodic iterate $\ubar^K$ and an arbitrary $u \in \reals^n$, we denote their subgradient of $g^\ast$ as $\pbar^K$ and~$p$, respectively:
\begin{equation}  \label{eq:drs-gf-pbar}
    \pbar^K \in \partial g^\ast (\ubar^K), \qquad p \in \partial g^\ast (u).
\end{equation}
One shall be aware that $\pbar^K$ is a subgradient of $g^\ast$ at $\ubar^K$ rather than the average of $\{p^k\}_{k=1}^K$. It then follows from \eqref{eq:fenchel} that $\ubar^K \in \partial g (\pbar^K)$, $u \in \partial g(p)$, and  
\[
    g^\ast(\ubar^K) = \langle \ubar^K, \pbar^K \rangle - g(\pbar^K), \qquad
    g^\ast(u) = \langle u, p \rangle - g(p).
\]
The introduction of $\pbar^K$ and $p$ helps reformulate the primal--dual gap function:
\begin{equation} \label{eq:drs-gf-gap}
    \cL(\xbar^K, u) - \cL(x, \ubar^K) = f(\xbar^K) + \langle u, \xbar^K \rangle - \langle u, p \rangle + g(p) - \big(f(x) + \langle \ubar^K, x \rangle - \langle \ubar^K, \pbar^K \rangle + g(\pbar^K) \big).
\end{equation}

We now derive an equality that will play an important role in obtaining the tight convergence rate of \eqref{eq:drs-gf}. Its development is motivated by a computer-aided analysis framework known as the performance estimation problem (PEP) \cite{DT14,THG17a}.
\begin{proposition} \label{prop:drs-gf-conv}
    Suppose $f$ and $g$ are \CCP\ functions, and $\{(x^k, u^k)\}_{k\in\bbN}$ is generated by \eqref{eq:drs-gf} with step size $\alpha>0$. Denote $p^k$ as in \eqref{eq:drs-gf-p}, and $\pbar^K$, $p$ as in \eqref{eq:drs-gf-pbar}. For $K \in \natint_+$, define the ergodic iterates
    \begin{equation} \label{eq:ergodic}
        \xbar^K \coloneqq \frac{1}{K} \sum_{k=1}^{K} x^k, \qquad 
        \ubar^K \coloneqq \frac{1}{K} \sum_{k=1}^{K} u^k.
   \end{equation}
    Then, for all $K \in \natint_+$ and all $x, u \in \reals^n$, the equality
    \begin{equation} \label{eq:drs-gf-conv-eq}
        \cL(\xbar^K,u) - \cL(x,\ubar^K) - \frac{1}{\alpha (K+1)} \pr{\|x^0 - x\|^2 + \alpha^2 \|u^0 - u\|^2} = \ineqs_f + \ineqs_g - \sos_1 - \sos_2
    \end{equation}
    holds, 
    where we denote 
    \begin{equation} \label{eq:drs-gf-conv-xdif}
        \xdif^k = \subg{f}(x^k) + u^k
    \end{equation}
    and    
    \begin{equation} \label{eq:drs-gf-conv-ineqs-squares}
        \begin{aligned}
            \ineqs_f &= \frac{1}{K} \sum_{k=1}^{K} \pr{ f ( \bar{x}^K ) - f(x^k) + \inner{\subg{f} ( \bar{x}^K )}{x^k - \bar{x}^K} } 
             +  \frac{1}{K} \sum_{k=1}^{K}  \pr{ f(x^{k}) - f(x) + \inner{\subg{f}(x^{k})}{x - x^{k}} } \\
             \ineqs_g &=  \frac{1}{K} \sum_{k=1}^{K} \pr{ g(p^k) - g( \pbar^K ) + \inner{u^k}{\pbar^K - p^k} } 
             +  \frac{1}{K} \sum_{k=1}^{K}  \pr{ g(p) - g(p^{k}) + \inner{u}{p^{k}-p} } \\
             \sos_1 &= \frac{1}{\alpha (K+1)} \pr{ \norm{ x^0 - x - \frac{\alpha (K+1)}{2K} \sum_{k=1}^{K} \xdif^k }^2 + \alpha^2 \norm{ u^0 - u - \frac{K+1}{2K} \sum_{k=1}^{K} \xdif^k }^2 } \\
             \sos_2 &= \frac{\alpha}{2K^2} \sum_{k=1}^{K} \sum_{l=1}^{k-1} \norm{ \xdif^k - \xdif^l }^2.
        \end{aligned}
    \end{equation}
\end{proposition}

Before proving \cref{prop:drs-gf-conv}, we present its two immediate implications. More specifically, \cref{prop:drs-gf-conv} is used to establish an ergodic convergence rate of \eqref{eq:drs-gf} (see \cref{thm:drs-gf-conv}) and also helpful in building a worst-case example (see \cref{cor:drs-gf-conv}).
\begin{theorem}[Convergence of \eqref{eq:drs-gf}] \label{thm:drs-gf-conv}
    Suppose $f$ and $g$ are \CCP\ functions, and $\{(x^k, u^k)\}_{k\in\bbN}$ is generated by \eqref{eq:drs-gf} with step size $\alpha>0$. Then, for all $K \in \natint_+$, the averaged iterates $(\xbar^K, \ubar^K)$ defined in~\eqref{eq:ergodic} satisfy
    \begin{equation} \label{eq:drs-gf-conv-gap}
        \cL(\xbar^K,u) - \cL(x,\ubar^K) \le \frac{1}{\alpha (K+1)} \pr{\|x^0 - x\|^2 + \alpha^2 \|u^0 - u\|^2}
    \end{equation}
    for all $x \in \dom f$ and all $u \in \dom g^\ast$.
\end{theorem}
\begin{proof}
    It follows from \eqref{eq:cvx} and the convexity of $f$ and $g$ that $\ineqs_f$ and $\ineqs_g$ defined in \eqref{eq:drs-gf-conv-ineqs-squares} are nonpositive. Moreover, the two quantities $\sos_1$ and $\sos_2$ are nonnegative since they are sums of squared terms.
    The desired conclusion \eqref{eq:drs-gf-conv-gap} then follows directly from \cref{prop:drs-gf-conv}.
\end{proof}

Moreover, the inequality \eqref{eq:drs-gf-conv-gap} holds with equality if and only if the four quantities defined in~\eqref{eq:drs-fg-conv-ineqs-squares} are all zero. This result is formalized in \cref{cor:drs-gf-conv} and will be used to build a worst-case example that demonstrates the tightness of the convergence rate \eqref{eq:drs-gf-conv-gap}.  
\begin{corollary} \label{cor:drs-gf-conv}
    Let $x \neq x^0$ and $u \neq u^0$. Under the same setting as in \cref{thm:drs-gf-conv}, the inequality~\eqref{eq:drs-gf-conv-gap} holds with equality if and only if the four quantities $\ineqs_f$, $\ineqs_g$, $\sos_1$, and $\sos_2$ are all zero.
\end{corollary}
\begin{proof}
    Recall from the proof of \cref{thm:drs-gf-conv} that $\ineqs_f$ and $\ineqs_g$ are nonpositive and $\sos_1$ and $\sos_2$ are nonnegative. This implies that $\ineqs_f + \ineqs_g - \sos_1 - \sos_2 = 0$ if and only if each term is zero. So, the desired conclusion follows directly from \cref{thm:drs-gf-conv}.
\end{proof}

Now, we prove \cref{prop:drs-gf-conv}. 
\begin{proof}[Proof of \cref{prop:drs-gf-conv}]
    It follows from \eqref{eq:drs-gf-gap} that the left-hand side of \eqref{eq:drs-gf-conv-eq} equals
    \begin{equation} \label{eq:drs-gf-lhs}
        \begin{aligned}
         \lhs &= f(\xbar^K) + \langle u, \xbar^K \rangle  + g(p) - \langle u, { p } \rangle 
         - \pr{ f(x) + \langle \ubar^K, x \rangle  + g(\pbar^K) - \langle \ubar^K, { \pbar^K } \rangle }  \\ &\phantom{=}
          - \frac{1}{\alpha (K+1)} \pr{ \|x^0 - x\|^2 + \alpha^2 \|u^0 - u\|^2 }.
        \end{aligned}
    \end{equation}
    To establish the identity \eqref{eq:drs-gf-conv-eq}, we simplify the four terms in \eqref{eq:drs-gf-conv-ineqs-squares} one by one. 
    For $\ineqs_f$, we observe that
    \begin{equation*}
        \frac{1}{K} \sum_{k=1}^{K} \inner{\subg{f} ( \bar{x}^K )}{x^k - \bar{x}^K} 
        =  \inner{\subg{f} ( \bar{x}^K )}{ \frac{1}{K} \sum_{k=1}^{K} x^k - \bar{x}^K} 
        = 0.
    \end{equation*}
    Note that the $f(x^k)$ terms cancel out, so $\ineqs_f$ simplifies to
    \begin{equation} \label{eq:drs-gf-ineqs-f}
        \ineqs_f = f(\xbar^K) - f(x) + \inner{ \frac{1}{K} \sum_{k=1}^{K} \subg{f}(x^k) }{x} - \frac{1}{K} \sum_{k=1}^{K} \inner{ \subg{f}(x^k) }{ x^k },
    \end{equation}
    in which the first two terms appear in the $\lhs$. 
    
    For $\ineqs_g$, we apply $\frac{1}{K} \sum_{k=1}^{K} \inner{u^k}{\pbar^{K}} = \inner{\ubar^K}{\pbar^{K}}$ and obtain
    \begin{equation} \label{eq:drs-gf-ineqs-g-1}
        \ineqs_g = g(p) - g(\pbar^K) - \inner{u}{p} + \inner{\ubar^K}{\pbar^{K}}  + \inner{u}{ \frac{1}{K} \sum_{k=1}^{K}  p^k} - \frac{1}{K} \sum_{k=1}^{K} \inner{u^k}{p^k}.
    \end{equation}
    Then it follows from \eqref{eq:drs-gf-reform-x}, \eqref{eq:drs-gf-p}, and \eqref{eq:drs-gf-conv-xdif} that 
    \begin{equation} \label{eq:drs-gf-eliminate-p}
        p^{k} 
        = x^{k} + \alpha \pr{ \subg f(x^{k}) +  u^{k}}
        = x^{k} + \alpha \xdif^k.
    \end{equation}
    Substituting back in \eqref{eq:drs-gf-ineqs-g-1} eliminates $\{p_k\}_{k=1}^K$ and thus $\ineqs_g$ becomes
    \begin{equation} \label{eq:drs-gf-ineqs-g-2}
        \ineqs_g = g(p) - g(\pbar^K) - \inner{u}{p} + \inner{\ubar^K}{\pbar^{K}}  + \inner{u}{\bar{x}^k} + \alpha \inner{u}{ \frac{1}{K} \sum_{k=1}^{K} v^k} - \frac{1}{K} \sum_{k=1}^{K} \inner{u^k}{x^{k} + \alpha \xdif^k}.
    \end{equation}
    Note that all terms appear in the \lhs\ except the last two terms.
    
    Similarly, for $\sos_1$, we expand and re-organize the squares and obtain 
    {\small
    \begin{equation*}
        \sos_1
        = \frac{1}{\alpha (K+1)} \pr{ \|x^0 - x\|^2 + \alpha^2 \|u^0 - u\|^2 } 
        - \inner{ x^0 + \alpha u^0 - x - \alpha u }{ \frac{1}{K} \sum_{k=1}^{K} \xdif^k }
        + \frac{\alpha (K+1)}{2K^2} \norm{ \sum_{k=1}^{K} \xdif^k }^2.
    \end{equation*}
    }
    Then, it is straightforward to verify that
    \begin{align}
        \MoveEqLeft[0.2] (\ineqs_f + \ineqs_g - \sos_1) - \lhs \nonumber \\
        &= \inner{ \frac{1}{K} \sum_{k=1}^{K} \subg{f}(x^k) }{x} - \frac{1}{K} \sum_{k=1}^{K} \inner{ \subg{f}(x^k) }{ x^k } 
        + \alpha \inner{u}{ \frac{1}{K} \sum_{k=1}^{K} v^k} - \frac{1}{K} \sum_{k=1}^{K} \inner{u^k}{x^{k} + \alpha \xdif^k} \nonumber \\
        &\phantom{=} + \inner{ x^0 + \alpha u^0 - x - \alpha u }{ \frac{1}{K} \sum_{k=1}^{K} \xdif^k } 
        - \frac{\alpha (K+1)}{2K^2} \norm{ \sum_{k=1}^{K} \xdif^k }^2 + \langle x, \ubar^K \rangle \nonumber \\
        &= - \frac{1}{K} \sum_{k=1}^{K} \inner{ \subg{f}(x^k) }{ x^k }
        - \frac{1}{K} \sum_{k=1}^{K} \inner{u^k}{x^k + \alpha \xdif^k}
        + \frac{1}{K} \sum_{k=1}^{K} \inner{ x^0 + \alpha u^0 }{  \xdif^k } 
        - \frac{\alpha (K+1)}{2K^2} \norm{ \sum_{k=1}^{K} \xdif^k }^2, \label{eq:drs-gf-without-S2-1}
    \end{align}
    where in the second equality we cancel out all the inner product terms involving $x$ or $u$. Then, it follows from \eqref{eq:drs-gf-conv-xdif} that the first two terms on the right-hand side of \eqref{eq:drs-gf-without-S2-1} simplifies to
    \begin{align*}
           - \frac{1}{K} \sum_{k=1}^{K} \inner{ \subg{f}(x^k) }{ x^k }           - \frac{1}{K} \sum_{k=1}^{K} \inner{u^k}{x^k - \alpha \xdif^k} 
           &= - \frac{1}{K} \sum_{k=1}^{K} \inner{ \subg{f}(x^k) + u^k }{ x^k }
           -  \frac{1}{K} \sum_{k=1}^{K} \inner{ \xdif^k}{  \alpha u^k} \\
           &= -\frac{1}{K} \sum_{k=1}^{K} \inner{ \xdif^k }{ x^k + \alpha u^k }.
    \end{align*}
    Substituting this into~\eqref{eq:drs-gf-without-S2-1} and reorganizing, we obtain
    \begin{equation} \label{eq:drs-gf-without-S2-2}
        (\ineqs_f + \ineqs_g - \sos_1) - \lhs
        = \frac{1}{K} \sum_{k=1}^{K} \inner{ \xdif^k }{ x^0 + \alpha u^0 - x^k - \alpha u^k } - \frac{\alpha (K+1)}{2K^2} \norm{ \sum_{k=1}^{K} \xdif^k }^2.
    \end{equation}
    Now we eliminate $x^k$. Applying \eqref{eq:drs-gf-reform-x} recursively, we obtain
    \begin{equation*}
        x^k = x^{k-1} -  \alpha (u^{k} - u^{k-1}) - \alpha \xdif^k
        = \dots = x^0 + \alpha u^0 - \alpha u^k - \alpha \sum_{l=1}^{k} \xdif^l.
    \end{equation*}
    Substituting it into \eqref{eq:drs-gf-without-S2-2} and using \eqref{eq:drs-gf-conv-xdif}, we obtain
    \begin{align}
        \MoveEqLeft[0.2] (\ineqs_f + \ineqs_g - \sos_1) - \lhs \nonumber \\
        &= \frac{\alpha}{K} \sum_{k=1}^{K} \inner{ \xdif^k }{ \sum_{l=1}^{k} \xdif^l } - \frac{\alpha (K+1)}{2K^2} \norm{ \sum_{k=1}^{K} \xdif^k }^2 \nonumber \\
        &= \frac{\alpha}{K} \sum_{k=1}^{K} \sum_{l=1}^{k-1} \inner{ \xdif^k }{  \xdif^l } + \frac{\alpha}{K} \sum_{k=1}^{K} \norm{ \xdif^k }^2  - \frac{\alpha (K+1)}{2K^2} \pr{ \sum_{k=1}^{K} \norm{ \xdif^k }^2 + 2 \sum_{k=1}^{K} \sum_{l=1}^{k-1} \inner{ \xdif^k }{  \xdif^l } } \nonumber \\
        &= \frac{\alpha}{2K^2} \pr{ -2\sum_{k=1}^{K} \sum_{l=1}^{k-1} \inner{ \xdif^k }{  \xdif^l } + (K-1) \sum_{k=1}^{K} \norm{ \xdif^k }^2  } \nonumber \\
        &= \frac{\alpha}{2K^2}  \pr{ -2\sum_{k=1}^{K} \sum_{l=1}^{k-1} \inner{ \xdif^k }{  \xdif^l } + \sum_{k=1}^{K} \sum_{l=1}^{k-1} \pr{ \norm{ \xdif^k }^2 +  \norm{ \xdif^l }^2 } } \nonumber \\
        &= \frac{\alpha}{2K^2} \sum_{k=1}^{K} \sum_{l=1}^{k-1} \norm{ \xdif^k - \xdif^l }^2
        = \sos_2. \label{eq:drs-gf-prf-last}
    \end{align}
    The second-to-last equality follows from $(K-1) \sum_{k=1}^{K} \|\xdif^k\|^2 = \sum_{k=1}^{K} \sum_{l=1}^{k-1} (\|\xdif^k\|^2 + \|\xdif^l\|^2)$, which can be verified by comparing the coefficients of each $\norm{\xdif^i}^2$ terms.  Therefore, $\ineqs_f + \ineqs_g - \sos_1 - \sos_2 = \lhs$, which is our desired conclusion. 
\end{proof}

The tightness of \eqref{eq:drs-gf-conv-gap} is now verified using a worst-case example motivated by \cref{cor:drs-gf-conv}.
\begin{theorem}[Worst-case example for \eqref{eq:drs-gf}] \label{thm:drs-gf-worst-case}
    Under the same setting as in \cref{thm:drs-gf-conv}, for any $K \in \natint_+$ and any $\alpha > 0$, there exist \CCP\ functions $f$ and $g$ and points $x^0, u^0, x, u \in \reals^n$ such that $\|x^0 - x\|^2 + \alpha^2 \|u^0 - u\|^2 = 1$ and
    \begin{equation*} 
        \cL(\xbar^K,u) - \cL(x,\ubar^K) = \frac{1}{\alpha(K+1)} \pr{ \|x^0 - x\|^2 + \alpha^2 \|u^0 - u\|^2 }.
    \end{equation*}
\end{theorem}
\begin{proof}
    Fix $K \in \natint_+$ and $\alpha > 0$. Let $e_0 \in \reals^n$ denote an arbitrary unit vector; that is, a vector with one entry equal to one and all others equal to zero. Define $x^0 = e_0 / \sqrt{2} \in \reals^n$, $u^0 = x^0 / \alpha \in \reals^n$, and $\xtilde = \utilde = 0 \in \reals^n$. Then, the initial condition holds: $\|x^0 - \xtilde\|^2 + \alpha^2 \|u^0 - \utilde\|^2 = 1$. Let
    \[
        f(x) = \frac{\sqrt 2}{\alpha (K+1)} \|x\|, \qquad g(x) = 0,
    \]
    (so $g^\ast (y) = \delta_{\{0\}} (y)$). Under this setup, \eqref{eq:drs-gf} generates the iterates
    \begin{equation} \label{eq:drs-gf-worst-case-iter}
        u^{k} = \begin{cases}
            \frac{1}{\alpha}x^0 & k=0 \\
            0 & k \ge 1,
        \end{cases} \qquad \quad x^{k+1} = \begin{cases}
            \prox_{\alpha f} (2x^0) & k=0 \\
            \prox_{\alpha f} (x^k)  & k \ge 1.
        \end{cases}
    \end{equation}
    The $x$-iteration is simply the proximal point method starting at $2x^0$. Then, from the definition of~$f$, we have
    \begin{equation*}
        \prox_{\alpha f} (y) 
        = \begin{cases}
            \pr{ \|{y}\| - \frac{\sqrt{2}}{K+1} } \frac{y}{\|{y}\|} & \text{if} \ \|y\| \ge \frac{\sqrt{2}}{K+1} \\
            0 & \text{otherwise.} 
        \end{cases}
    \end{equation*}
    So, with $x^0 = e_0 / \sqrt{2}$, we show that
    \begin{equation*}
        x^k = \sqrt{2} \pr{ 1 - \frac{k}{K+1} } e_0, \qquad k = 1,\ldots,K
    \end{equation*}
    by induction.
    \begin{enumerate}
        \item [(i)] When $k=1$, it follows from $\frac{\sqrt{2}}{K+1} \le \sqrt{2} = \|2x^0\|$ that
        \begin{equation*}
            x^{1} 
            = \pr{ \|2x^0\| - \frac{\sqrt{2}}{K+1} } \frac{2x^0}{\|2x^0\|}
            = \sqrt{2} \pr{ 1 - \frac{1}{K+1} } e_0. 
        \end{equation*}

        \item [(ii)] Assume that the induction hypothesis is true for $k=m\le K-1$. Then, by the induction hypothesis, we have
        \[
            \frac{\sqrt{2}}{K+1} 
            \le \sqrt{2} \pr{ 1 - \frac{m}{K+1} }
            = \|x^m\|.
        \]
        Thus, 
        \begin{equation*}
            \pr{ \|{x^m}\| -  \frac{\sqrt{2}}{K+1} } \frac{x^m}{\|{x^m}\|} 
            = \pr{ \sqrt{2} \pr{ 1 - \frac{m}{K+1} } -  \frac{\sqrt{2}}{K+1} } e_0
            = \sqrt{2} \pr{1 - \frac{m+1}{K+1} } e_0,
        \end{equation*}
        so we can conclude with the desired result.
    \end{enumerate}
    Finally, invoking \cref{cor:drs-gf-conv}, it remains to prove that the four quantities $(\ineqs_f, \ineqs_g, \sos_1, \sos_2)$ in \eqref{eq:drs-gf-conv-ineqs-squares} are zero. Observe that the points $x^1,x^2,\ldots,x^K,\xbar^K,\xtilde$ lie in a line and that $\partial f(x^k)$ is a singleton for all $k=1,2,\ldots,K$. It implies that $\ineqs_f=0$. Next, we see that $\ineqs_g = 0$ since $g = 0$, $\utilde = 0$, and $u^k = 0$ for all $k=1,\ldots,K$. In addition, since $\xdif^k = \subg f(x^k) + u^k = \subg f(x^k) = \tfrac{\sqrt{2}}{\alpha(K+1)} e_0$ is a constant vector for all $k=1,\ldots,K$, we obtain $\sos_2=0$. Finally, from
    \begin{equation*}
        \frac{\alpha (K+1)}{2K} \sum_{k=1}^{K} \xdif^k
        = \frac{\alpha (K+1)}{2K} \sum_{k=1}^{K} \frac{\sqrt{2}}{\alpha(K+1)} e_0
        = \frac{1}{\sqrt 2} e_0
        = x^0 - \xtilde
        = \alpha ( u^0 - \utilde ),
    \end{equation*}
    we obtain $\sos_1=0$. This completes the proof.
\end{proof}

\subsection{\texorpdfstring{DRS-$fg$}{DRS-fg}: worst-case rate and its tightness} \label{sec:drs-fg}

We now proceed with the convergence analysis of \eqref{eq:drs-fg}. Following the same approach as in the beginning of \cref{sec:drs-gf}, we reformulate \eqref{eq:drs-fg} as
\begin{subequations}
    \begin{align}
        x^{k+1} &= x^k - \alpha u^k - \alpha \subg f(x^{k+1}) \label{eq:drs-fg-reform-x} \\ 
        u^{k+1} &= u^k + \tfrac{1}{\alpha} (2x^{k+1} - x^{k}) - \tfrac{1}{\alpha} \subg g^\ast (u^{k+1}),
    \end{align}
\end{subequations}
and similarly we define
\begin{equation} \label{eq:drs-fg-p}
    p^{k+1} := \subg g^\ast (u^{k+1}) = x^{k+1} + (x^{k+1} - x^{k}) + \alpha (u^{k} - u^{k+1}),
\end{equation}
Just as for \eqref{eq:drs-gf}, we introduce an equality that enables us to derive the tight convergence rate of \eqref{eq:drs-fg}.
\begin{proposition} \label{prop:drs-fg-conv}
    Suppose $f$ and $g$ are \CCP\ functions, and $\{(x^k, u^k)\}_{k\in\bbN}$ is generated by \eqref{eq:drs-fg} with step size $\alpha>0$. Denote $\xbar^K, \ubar^K, \pbar^K, p$ as in \Cref{prop:drs-gf-conv}, and $p^k$ as in \eqref{eq:drs-fg-p}. Then, for all $K \in \natint_+$ and all $x, u \in \reals^n$, the equality
    \begin{equation} \label{eq:drs-fg-conv-eq}
        \cL(\xbar^K,u) - \cL(x,\ubar^K) - \frac{1}{\alpha (K+1)} \pr{\|x^0 - x\|^2 + \alpha^2 \|u^0 - u\|^2} = \ineqs_f + \ineqs_g - \sos_1 - \sos_2
    \end{equation}
    holds, where $\xdif^k$ is defined in \eqref{eq:drs-gf-conv-xdif}, $(\ineqs_f, \ineqs_g, \sos_2)$ are defined in \eqref{eq:drs-gf-conv-ineqs-squares}, and $\sos_1$ is redefined as
    \begin{equation} \label{eq:drs-fg-conv-ineqs-squares}
        \begin{aligned}
             \sos_1 &:= \frac{1}{\alpha (K+1)} \pr{ \norm{ x^0 - x - \frac{\alpha (K+1)}{2K} \sum_{k=1}^{K} \xdif^k }^2 + \alpha^2 \norm{ u^0 - u + \frac{K+1}{2K} \sum_{k=1}^{K} \xdif^k }^2 }. %\\
        \end{aligned}
    \end{equation}
\end{proposition}
We point out that \eqref{eq:drs-fg-conv-ineqs-squares} differs from $\sos_1$ in \cref{prop:drs-gf-conv} by only a sign. This similarity suggests that the proof of \cref{prop:drs-fg-conv} closely parallels that of \cref{prop:drs-gf-conv}. Before presenting the proof, we emphasize that \cref{prop:drs-fg-conv} also serves as the key to establishing the tight convergence rate of \eqref{eq:drs-fg}.
\begin{theorem}[Convergence of \eqref{eq:drs-fg}] \label{thm:drs-fg-conv}
    Suppose $f$ and $g$ are \CCP\ functions, and $\{(x^k, u^k)\}_{k\in\bbN}$ is generated by \eqref{eq:drs-fg} with step size $\alpha>0$. Then, for all $k \in \natint_+$, the averaged iterates $(\xbar^K, \ubar^K)$ defined in \eqref{eq:ergodic} satisfy
    \begin{equation} \label{eq:drs-fg-conv-gap}
        \cL(\xbar^K,u) - \cL(x,\ubar^K) \le \frac{1}{\alpha (K+1)} \pr{\|x^0 - x\|^2 + \alpha^2 \|u^0 - u\|^2}
    \end{equation}
    for all $x \in \dom f$ and all $u \in \dom g^\ast$.
\end{theorem}
\begin{proof}
    It follows from \eqref{eq:cvx} and the convexity of $f$ and $g$ that the two quantities $\ineqs_f$ and $\ineqs_g$ defined in \eqref{eq:drs-fg-conv-ineqs-squares} are nonpositive. Moreover, the two quantities $\sos_1$ and $\sos_2$ are nonnegative since they are sums of square terms. So, the desired conclusion \eqref{eq:drs-fg-conv-gap} follows directly from \cref{prop:drs-fg-conv}.
\end{proof}

\Cref{prop:drs-fg-conv} also provides an explicit if-and-only-if condition that guides the construction of a worst-case example. We will leverage this corollary in the proof of \cref{thm:drs-fg-worst-case}. 
\begin{corollary} \label{cor:drs-fg-conv}
    Let $x \neq x^0$ and $u \neq u^0$. Under the same setting as in \cref{thm:drs-fg-conv}, the inequality~\eqref{eq:drs-fg-conv-gap} holds with equality if and only if the four quantities $\ineqs_f$, $\ineqs_g$, $\sos_1$, and $\sos_2$ are all zero.
\end{corollary}
\begin{proof}
    Recall from the proof of \cref{thm:drs-fg-conv} that $\ineqs_f$ and $\ineqs_g$ are nonpositive and $\sos_1$ and $\sos_2$ are nonnegative. This implies that $\ineqs_f + \ineqs_g - \sos_1 - \sos_2 = 0$ if and only if each term is zero. So, the desired conclusion follows directly from \cref{prop:drs-fg-conv}.
\end{proof}

Now, we prove \cref{prop:drs-fg-conv}. 
\begin{proof}[Proof of \cref{prop:drs-fg-conv}]
    Denote the left-hand side of \eqref{eq:drs-fg-conv-eq} as \lhs, as in \eqref{eq:drs-gf-lhs}. 
    The first step in organizing $\ineqs_f$ and $\ineqs_g$ can be done in the same way as in \eqref{eq:drs-gf-ineqs-f} and \eqref{eq:drs-gf-ineqs-g-1}. Now we eliminate the $\{p^k\}_{k=1}^K$ terms. We obtain from \eqref{eq:drs-fg-reform-x} that $x^{k} - x^{k-1} = - \alpha u^{k-1} - \alpha \subg f(x^{k})$. Substituting this into \eqref{eq:drs-fg-p} and recalling the definition of $\xdif^k$ in \eqref{eq:drs-gf-conv-xdif} yields
    \begin{equation} \label{eq:drs-fg-eliminate-p}
        p^k = x^k - \alpha (\subg{f}(x^k) + u^k) = x^k - \alpha \xdif^k. 
    \end{equation}
    Note that the sign of $\xdif^k$ is flipped compared to \eqref{eq:drs-gf-eliminate-p}.  
    Then, eliminating $p^k$ in $\ineqs_g$ yields 
    \begin{equation*}
        \ineqs_g = g(p) - g(\pbar^K) - \inner{u}{p} + \inner{\ubar^K}{\pbar^{K}}  + \inner{u}{\xbar^K} - \alpha \inner{u}{ \frac{1}{K} \sum_{k=1}^{K} \xdif^k } - \frac{1}{K} \sum_{k=1}^{K} \inner{u^k}{x^k - \alpha \xdif^k},
    \end{equation*}
    where all terms, except the last two, appear in the \lhs. Then, proceeding with a similar calculation to that used to obtain \eqref{eq:drs-gf-without-S2-1}, but being careful with the signs of $u$, $u^0$, and $\xdif^k$, we obtain:
    \begin{align}
        \MoveEqLeft[0.2] (\ineqs_f + \ineqs_g - \sos_1) - \lhs \nonumber \\
        &= - \frac{1}{K} \sum_{k=1}^{K} \inner{ \subg{f}(x^k) }{ x^k }
        - \frac{1}{K} \sum_{k=1}^{K} \inner{u^k}{x^k - \alpha \xdif^k}
        + \frac{1}{K} \sum_{k=1}^{K} \inner{ x^0 - \alpha u^0 }{  \xdif^k } 
        - \frac{\alpha (K+1)}{2K^2} \norm{ \sum_{k=1}^{K} \xdif^k }^2. \label{eq:drs-fg-without-S2-1}
    \end{align}
    Then, it follows from the definition of $v^k$ \eqref{eq:drs-gf-conv-xdif} that the first two terms on the right-hand side of~\eqref{eq:drs-fg-without-S2-1} simplifies to
    \begin{equation*}
    \begin{aligned}
         - \frac{1}{K} \sum_{k=1}^{K} \inner{ \subg{f}(x^k) }{ x^k } - \frac{1}{K} \sum_{k=1}^{K} \inner{u^k}{x^k - \alpha \xdif^k} 
        &= - \frac{1}{K} \sum_{k=1}^{K} \inner{ \subg{f}(x^k) + u^k }{ x^k }
        +  \frac{1}{K} \sum_{k=1}^{K} \inner{ \xdif^k}{  \alpha u^k} \\
        &= \frac{1}{K} \sum_{k=1}^{K} \inner{ \xdif^k }{ -x^k + \alpha u^k }.
    \end{aligned}
    \end{equation*}
    Substituting this into~\eqref{eq:drs-fg-without-S2-1} gives
    \begin{equation} \label{eq:drs-fg-without-S2-2}
        (\ineqs_f + \ineqs_g - \sos_1) - \lhs
        = \frac{1}{K} \sum_{k=1}^{K} \inner{ \xdif^k }{ x^0 - \alpha u^0 - x^k + \alpha u^k } - \frac{\alpha (K+1)}{2K^2} \norm{ \sum_{k=1}^{K} \xdif^k }^2.
    \end{equation}
    Now we eliminate $x^k$. Applying \eqref{eq:drs-fg-reform-x} recursively, we obtain
    \begin{equation*}
        x^k
        = \dots
        = x^0 - \alpha \sum_{l=0}^{k-1} \pr{ u^{l} + \subg{f}(x^{l+1}) } 
        = x^0 - \alpha u^0 - \alpha \subg{f}(x^k) - \alpha \sum_{l=1}^{k-1} \xdif^l.
    \end{equation*}
    Substituting it into \eqref{eq:drs-fg-without-S2-2} and using \eqref{eq:drs-gf-conv-xdif}, proceeding same calculation as in \eqref{eq:drs-gf-prf-last} we obtain
    \begin{equation*} 
            (\ineqs_f + \ineqs_g - \sos_1) - \lhs
            = \frac{\alpha}{K} \sum_{k=1}^{K} \inner{ \xdif^k }{ \sum_{l=1}^{k} \xdif^l } - \frac{\alpha (K+1)}{2K^2} \norm{ \sum_{k=1}^{K} \xdif^k }^2 
            = \sos_2.
    \end{equation*}
    Therefore, $\ineqs_f + \ineqs_g - \sos_1 - \sos_2 = \lhs$, which is our desired conclusion. 
\end{proof}

The tightness of \eqref{eq:drs-fg-conv-gap} is now verified using a worst-case example motivated by \cref{cor:drs-fg-conv}.
\begin{theorem}[Worst-case example for \eqref{eq:drs-fg}] \label{thm:drs-fg-worst-case}
    Under the same setting as in \cref{thm:drs-fg-conv}, for any $K \in \natint_+$ and any $\alpha > 0$, there exist \CCP\ functions $f$ and $g$ and points $x^0, u^0, x, u \in \reals^n$ such that $\|x^0-x\|^2 + \alpha^2 \|u^0 - u\|^2 = 1$ and
    \begin{equation*} 
        \cL(\xbar^K,u) - \cL(x,\ubar^K) = \frac{1}{\alpha(K+1)} \pr{ \|x^0 - x\|^2 + \alpha^2 \|u^0 - u\|^2 }.
    \end{equation*}
\end{theorem}
\begin{proof}
    Fix $K \in \natint_+$ and $\alpha > 0$. Let $e_0 \in \reals^n$ denote an arbitrary unit vector; that is, a vector with one entry equal to one and all others equal to zero. Define $x^0 = e_0 / \sqrt{2} \in \reals^n$, $u^0 = -x^0 / \alpha \in \reals^n$, and $\xtilde = \utilde =0 \in \reals^n$. Then, the initial condition holds: $\|x^0 - \xtilde\|^2 + \alpha^2 \|u^0 - \utilde\|^2 = 1$. Let 
    \begin{equation*}
        f(x) = \frac{\sqrt{2}}{{\alpha}(K+1)} \norm{ x }, \qquad g(x) = 0,
    \end{equation*}
    (so $g^\ast (y) = \delta_{\set{0}}(y)$). Under this setup, \eqref{eq:drs-fg} generates the iterates
    \begin{equation*}
        u^{k} = \begin{cases}
            -\frac{1}{\alpha} x^0 & k=0 \\
            0 & k \ge 1,
        \end{cases} \qquad \quad x^{k+1} = \begin{cases}
            \prox_{\alpha f} (2x^0) & k=0 \\
            \prox_{\alpha f} (x^k) & k\ge 1.
        \end{cases}
    \end{equation*}
    Comparison with \eqref{eq:drs-gf-worst-case-iter} reveals that \eqref{eq:drs-fg} generates the same sequence $\{(x^k,u^k)\}$ despite a different initial point $u^0$. Since the $x$-iterates are exactly the same as in the proof of \cref{thm:drs-gf-worst-case}, the remainder of the proof readily extends from that of \cref{thm:drs-gf-worst-case}.
\end{proof}

As discussed earlier, \eqref{eq:drs-fg} and \eqref{eq:drs-gf} are not equivalent, in the sense that they generally produce different sequences of iterates. However, as shown in \cref{thm:drs-gf-worst-case} and \cref{thm:drs-fg-worst-case}, their tight convergence rates are identical (under the general convex setting). Moreover, the corresponding worst-case examples are nearly the same, differing only in the sign of the initial iterate $u^0$.

%%%%%%%%%%%%%%%%%%%%%%%%%%%%%%%%%%%%%%%%%%%%%%%%%%%%%%%%%%%%%%%%%%%%%%%%%%%%%%%%%%%%%%%%%%%%%%%%%%%%%%%%%%%%%%
\section{Convergence analysis of two variants of DYS} \label{sec:dys}

\Cref{sec:drs} derives an $\cO(1/(K+1))$ ergodic rate of convergence for both variants of DRS. Yet, the known ergodic rate for both \eqref{eq:dys-gf} and \eqref{eq:dys-fg} is $\cO(1/K)$ \cite{JV23,SCMR22,Yan18}, of which the tightness is not addressed. So in this section, we investigate the convergence rates of both DYS variants. Interestingly, \eqref{eq:dys-gf} (originally proposed in \cite{DY17b}) has an $\cO(1/K)$ ergodic rate, slower than its special case \eqref{eq:drs-gf}, whereas the swapped version \eqref{eq:dys-fg} restores the $\cO(1/(K+1))$ rate as in \eqref{eq:drs-fg}.

Again, our analysis uses the primal--dual gap function, which, from the definition of $p$ and $\pbar^K$ in~\eqref{eq:drs-gf-pbar}, can be reformulated as
\begin{equation} \label{eq:dys-gap}
    \begin{split}
        \cL(\bar{x}^K, u) - \cL(x, \bar{u}^K) 
        &= f(\bar{x}^K) + h(\bar{x}^K) + \langle u, \bar{x}^K \rangle - \langle u, p \rangle + g(p)  \\ &\phantom{=} - \big(f(x) + h(x) + \langle \bar{u}^K, x \rangle - \langle \bar{u}^K, \pbar^K \rangle + g(\pbar^K) \big).
    \end{split}
\end{equation}

\subsection{Analysis of \texorpdfstring{DYS-$gf$}{DYS-gf}}

In parallel to \cref{sec:drs-gf}, we reformulate \eqref{eq:dys-gf} as
\begin{subequations} \label{eq:dys-gf-reform}
    \begin{align}
        u^{k+1} &= u^k + \tfrac{1}{\alpha} x^k - \tfrac{1}{\alpha} \subg g^\ast (u^{k+1}) \label{eq:dys-gf-reform-u} \\
        p^{k+1} &:= \subg g^\ast (u^{k+1}) =  \alpha u^{k} + x^{k} - \alpha u^{k+1} \label{eq:dys-gf-p} \\
        x^{k+1} &= x^k - \alpha (2u^{k+1} - u^k) - \alpha \nabla h(p^{k+1}) - \alpha \subg f(x^{k+1}) \label{eq:dys-gf-reform-x-1} \\
        &= p^{k+1} - \alpha u^{k+1} - \alpha \nabla h(p^{k+1}) - \alpha \subg f(x^{k+1}) \label{eq:dys-gf-reform-x-2}.
    \end{align}
\end{subequations}
Not surprisingly, this iteration reduces to \eqref{eq:drs-gf-reform} when $h=0$.

We now prove the core equality that provides the convergence proof. Remarkably, \cref{prop:dys-gf-conv} does not reduce to \cref{prop:drs-gf-conv} when $h=0$. 
\begin{proposition} \label{prop:dys-gf-conv}
    Suppose $f$, $g$, and $h$ are CCP and $h$ is $L$-smooth (with $L>0$). Suppose also that $\{(x^k,u^k)\}_{k \in \natint}$ is generated by \eqref{eq:dys-gf} with step size $\alpha = \frac{1}{L}$. Denote $p^k$ as in \eqref{eq:drs-gf-p}, $\pbar^K$, $p$ as in \eqref{eq:drs-gf-pbar}, and $(\xbar^K, \ubar^K)$ as in \eqref{eq:ergodic}. Then, for all $K \in \mathbb{N}_{+}$ and all $x,u \in \reals^n$, the following equality holds
    \begin{equation*} \label{eq:dys-gf-conv-eq}
            \cL(\xbar^K,u) - \cL(x,\ubar^K) - \frac{1}{\alpha K} \pr{ \|x^0 - x\|^2 + \alpha^2 \|u^0 - u\|^2 }  = \ineqs_f + \ineqs_g + \ineqs_h - \sos_h - \sos_1 - \sos_2,
    \end{equation*}
    where 
    \begin{equation} \label{eq:dys-gf-conv-xdif}
        \xdif^k =  \subg f(x^{k}) +  u^{k} + \nabla h(p^{k}),
    \end{equation}
    $(\ineqs_f, \ineqs_g)$ are defined as in \eqref{eq:drs-gf-conv-ineqs-squares}, and
    {\allowdisplaybreaks
        \begin{align*}
        \ineqs_h 
             &:= \frac{1}{K} \sum_{k=1}^{K} \pr{ h ( \bar{x}^K ) - h(p^{k}) + \inner{\nabla h ( \bar{x}^K )}{p^{k} - \bar{x}^K} + \frac{\alpha}{2} \norm{ \nabla h(\bar{x}^{K}) -  \nabla h(p^{k}) }^2 }  \\ &\phantom{=}
             +   \frac{1}{K} \sum_{k=1}^{K}  \pr{ h(p^{k}) - h(x) + \inner{\nabla h(p^{k})}{x - p^{k}}  + \frac{\alpha}{2} \norm{ \nabla h(x) -\nabla h(p^{k}) }^2 } \\ 
        \sos_h 
             &:= \frac{\alpha}{2} \norm{ \nabla h(\bar{x}^K) + \frac{1}{K} \sum_{k=1}^{K} \pr{ \subg f(x^k) + u^k }  }^2 
            + \frac{\alpha}{2} \norm{ \nabla h(x) - \frac{1}{K} \sum_{k=1}^{K} \nabla h(p^{k}) }^2 \\
        \sos_1 
            &:= \frac{1}{\alpha K} \Bigg(  \norm{ x^0 - x   - \frac{\alpha }{2} \sum_{k=1}^{K} \xdif^k }^2 
            + \alpha^2 \norm{ u^0 - u   - \frac{1}{2} \sum_{k=1}^{K}\xdif^k  }^2 \Bigg) \\
        \sos_2
            &:= \frac{\alpha}{2K^2} \sum_{k=1}^{K} \sum_{l=1}^{k-1}  \pr{ \norm{ v^k - \nabla h(p^k) - (v^l - \nabla h(p^l)) }^2 + \norm{ \nabla h(p^{k}) - \nabla h(p^{l}) }^2   }.
        \end{align*}
    }
\end{proposition}
In \cref{prop:dys-gf-conv}, the step size is set to $\alpha = \tfrac{1}{L}$ for two main reasons. First, this choice simplifies presentation of the analysis. The quantities $(\ineqs_h,\sos_h,\sos_1,\sos_2)$ are already intricate owing to the presence of the smooth term $h$; allowing for a broader range of step sizes would further complicate the presentation and proofs with limited additional insight. Second, unlike DRS, the step size $\alpha$ in \eqref{eq:dys-gf} must be upper bounded by a function of $L$, and the exact admissible range remains unclear. A recent paper \cite{AT22} explores ways to enlarge this range, but the question is still open. Given these considerations, we believe the simplified setting $\alpha = \tfrac{1}{L}$ is sufficient for the purpose of this paper.

\begin{proof}[Proof of \cref{prop:dys-gf-conv}]
It follows from \eqref{eq:dys-gf-reform-x-2} and \eqref{eq:dys-gf-conv-xdif} that
\begin{equation} \label{eq:dys-gf-eliminate-p}
    p^{k} 
    = x^{k} + \alpha (\subg f(x^{k}) +  u^{k} + \nabla h(p^{k}))
    = x^{k} + \alpha \xdif^k.
\end{equation}
Recalling \eqref{eq:dys-gap}, left-hand side of \eqref{eq:dys-gf-conv-eq} is
\begin{equation*}
    \begin{aligned}
     \lhs &= f(\xbar^K) + h(\xbar^K) + \langle u, \xbar^K \rangle  + g(p) - \langle u, { p } \rangle 
     - \pr{ f(x) + h(x) + \langle \ubar^K, x \rangle  + g(\pbar^K) - \langle \ubar^K, { \pbar^K } \rangle }  \\ &\phantom{=}
      - \frac{1}{\alpha K} \pr{ \|x^0-x\|^2 + \alpha^2 \|u^0 - u\|^2 }.
    \end{aligned}
\end{equation*}
With the same argument as in \Cref{prop:drs-gf-conv}, $\ineqs_f$ and $\ineqs_g$ simplify to
\begin{equation} \label{eq:dual_dys_equality_lhs}
    \begin{aligned}
        \ineqs_f 
        &= f(\xbar^K) - f(x) + \inner{ x }{ \frac{1}{K} \sum_{k=1}^{K} \subg{f}(x^k) } - \frac{1}{K} \sum_{k=1}^{K} \inner{ \subg{f}(x^k) }{ x^k } \\
        \ineqs_g 
        &= g(p) - g(\pbar^K) - \inner{u}{p} + \inner{\ubar^K}{\pbar^{K}}  + \inner{u}{\xbar^K} + \alpha \inner{u}{ \frac{1}{K} \sum_{k=1}^{K} \xdif^k } - \frac{1}{K} \sum_{k=1}^{K} \inner{u^k}{x^k + \alpha \xdif^k}.
    \end{aligned}
\end{equation}

To regroup some of the terms, we define
\begin{equation*}
    \begin{aligned}
    \tildeineqs_h 
         &= \ineqs_h - \frac{\alpha}{2K} \sum_{k=1}^{K} \norm{ \nabla h(\bar{x}^{K}) -  \nabla h(p^{k}) }^2 -  \frac{\alpha}{2K} \sum_{k=1}^{K} \norm{ \nabla h(x) -\nabla h(p^{k}) }^2 \\
    \tildesos_h 
         &= \sos_h -\frac{\alpha}{2K} \sum_{k=1}^{K} \norm{ \nabla h(\bar{x}^{K}) -  \nabla h(p^{k}) }^2 -  \frac{\alpha}{2K} \sum_{k=1}^{K} \norm{ \nabla h(x) -\nabla h(p^{k}) }^2.
    \end{aligned}
\end{equation*}
We can easily verify that $\tildeineqs_h-\tildesos_h = \ineqs_h-\sos_h$. Next, we simplify $\tildeineqs_h$, $\tildesos_h$, and $\sos_1$ one by one.

For $\tildeineqs_h$, it follows from \eqref{eq:dys-gf-eliminate-p} and $\xbar^K = \frac{1}{K} \sum_{k=1}^{K} x^k$ that
\begin{equation*}   
    \frac{1}{K} \sum_{k=1}^{K} \inner{\nabla h ( \bar{x}^K )}{p^{k} - \bar{x}^K}
    = \frac{1}{K} \sum_{k=1}^{K} \inner{\nabla h ( \bar{x}^K )}{x^{k} + \alpha \xdif^k - \bar{x}^K}
    = \frac{\alpha}{K} \sum_{k=1}^{K} \inner{\nabla h ( \bar{x}^K )}{\xdif^k}.
\end{equation*}
Applying \eqref{eq:dys-gf-eliminate-p} to eliminate $p^k$, we can verify $\tildeineqs_h$ simplifies to
\begin{equation} \label{eq:dys-gf-Ih}
    \tildeineqs_h = h(\xbar^K) - h(x) + \frac{\alpha}{K} \sum_{k=1}^{K} \inner{\nabla h(\bar{x}^{K})}{\xdif^{k}} + \frac{1}{K} \sum_{k=1}^{K} \inner{\nabla h(p^{k})}{x - x^k - \alpha \xdif^k  },
\end{equation}
where first two terms appears in the \lhs.

Similarly, we have for $\tildesos_h$ that
\begin{align}
    \tildesos_h &=
        - \frac{\alpha}{2K} \sum_{k=1}^{K} \norm{ \nabla h(x) -\nabla h(p^{k}) }^2 + \frac{\alpha}{2} \norm{ \nabla h(x) - \frac{1}{K} \sum_{k=1}^{K} \nabla h(p^{k}) }^2 \nonumber \\
    &\phantom{=}
        -\frac{\alpha}{2K} \sum_{k=1}^{K} \norm{ \nabla h(\bar{x}^{K}) -  \nabla h(p^{k}) }^2 
        + \frac{\alpha}{2} \norm{ \nabla h(\bar{x}^K) + \frac{1}{K} \sum_{k=1}^{K} \pr{ \subg f(x^k) + u^k } }^2 \nonumber \\
    &= 
        - \frac{\alpha}{2K} \sum_{k=1}^K \norm{ \nabla h(p^k) }^2 + \frac{\alpha}{2} \norm{ \frac{1}{K} \sum_{k=1}^{K} \nabla h(p^k) }^2 \nonumber \\
    &\phantom{=}
        -\frac{\alpha}{2K} \sum_{k=1}^{K} \norm{ \nabla h(\bar{x}^{K}) -  \nabla h(p^{k}) }^2 
        + \frac{\alpha}{2} \norm{ \nabla h(\bar{x}^K) + \frac{1}{K} \sum_{k=1}^{K} \pr{ \subg f(x^k) + u^k } }^2. \label{eq:dys-gf-Sh-1}
\end{align}
The last two terms on the right-hand side of \eqref{eq:dys-gf-Sh-1} can be further simplified to
\begin{align*}
    &-\frac{\alpha}{2K} \sum_{k=1}^{K} \norm{ \nabla h(\bar{x}^{K}) -  \nabla h(p^{k}) }^2 + \frac{\alpha}{2} \norm{ \nabla h(\bar{x}^K) + \frac{1}{K} \sum_{k=1}^{K} \pr{ \subg f(x^k) + u^k }  }^2   \\
    &= - \frac{\alpha}{2K} \sum_{k=1}^K \norm{ \nabla h(p^k) }^2 + \frac{\alpha}{K} \sum_{k=1}^K \inner{\nabla h(\bar{x}^{K})}{\nabla h(p^k) + \subg f(x^k) + u^k}  + \frac{\alpha}{2} \norm{ \frac{1}{K} \sum_{k=1}^{K} \pr{ \subg f(x^k) + u^k } }^2 \\ 
    &= - \frac{\alpha}{2K} \sum_{k=1}^K \norm{ \nabla h(p^k) }^2 + \frac{\alpha}{K} \sum_{k=1}^K \inner{\nabla h(\bar{x}^{K})}{\xdif^k}  + \frac{\alpha}{2} \norm{ \frac{1}{K} \sum_{k=1}^{K} \pr{ \subg f(x^k) + u^k } }^2.
\end{align*}
Substituting it back to \eqref{eq:dys-gf-Sh-1} gives
\begin{equation} \label{eq:dys-gf-Sh-2}
    \resizebox{0.98\textwidth}{!}{$\displaystyle
    \tildesos_h 
    = - \frac{\alpha}{K} \sum_{k=1}^K \norm{ \nabla h(p^k) }^2 + \frac{\alpha}{2} \norm{ \frac{1}{K} \sum_{k=1}^{K} \pr{ \subg f(x^k) + u^k } }^2  + \frac{\alpha}{2} \norm{ \frac{1}{K} \sum_{k=1}^{K} \nabla h(p^k) }^2
    + \frac{\alpha}{K} \sum_{k=1}^K \inner{\nabla h(\bar{x}^{K})}{\xdif^k}.$}
\end{equation}

Next, $\sos_1$ can be simplified to
\begin{equation} \label{eq:dys-gf-S1}
    \sos_1
    = \frac{1}{\alpha K} \pr{ \|x^0 - x\|^2 + \alpha^2 \|u^0 - u\|^2 } 
    - \inner{ x^0 + \alpha u^0 - x  - \alpha u }{ \frac{1}{K} \sum_{k=1}^{K} \xdif^k }
    + \frac{\alpha}{2K} \norm{ \sum_{k=1}^{K} \xdif^k }^2.
\end{equation}

Combining \eqref{eq:dys-gf-Ih}, \eqref{eq:dys-gf-Sh-2} and \eqref{eq:dys-gf-S1} yields
\begin{align}
    \MoveEqLeft[0.2] (\ineqs_f + \ineqs_g + \ineqs_h - \sos_h - \sos_1) - \lhs \nonumber \\
    &= (\ineqs_f + \ineqs_g + \tildeineqs_h - \tildesos_h - \sos_1) - \lhs \nonumber \\
    &= 
    \inner{ x }{ \frac{1}{K} \sum_{k=1}^{K} \subg{f}(x^k) } - \frac{1}{K} \sum_{k=1}^{K} \inner{ \subg{f}(x^k) }{ x^k } 
     +  \alpha \inner{u}{ \frac{1}{K} \sum_{k=1}^{K} \xdif^k } - \frac{1}{K} \sum_{k=1}^{K} \inner{u^k}{x^k + \alpha \xdif^k} \nonumber \\ 
    &\phantom{=}
    + \frac{1}{K} \sum_{k=1}^{K} \inner{\nabla h(p^{k})}{x - x^k - \alpha \xdif^k  } + \frac{\alpha}{K} \sum_{k=1}^K \norm{ \nabla h(p^k) }^2 
    - \frac{\alpha}{2} \norm{ \frac{1}{K} \sum_{k=1}^{K} \pr{ \subg f(x^k) + u^k } }^2 \nonumber \\
    &\phantom{=}
    - \frac{\alpha}{2} \norm{ \frac{1}{K} \sum_{k=1}^{K} \nabla h(p^k) }^2
    + \inner{ x^0  +  \alpha u^0 - x - \alpha u }{ \frac{1}{K} \sum_{k=1}^{K} \xdif^k }
    - \frac{\alpha}{2K} \norm{ \sum_{k=1}^{K} \xdif^k }^2 + \langle \bar{u}^K, x \rangle \nonumber \\
    &=
    - \frac{1}{K} \sum_{k=1}^{K} \inner{ \subg{f}(x^k) }{ x^k }
    - \frac{1}{K} \sum_{k=1}^{K} \inner{u^k + \nabla h(p^k)}{x^k + \alpha \xdif^k} 
    + \frac{\alpha}{K} \sum_{k=1}^{K} \norm{  \nabla h(p^k) }^2 
    - \frac{\alpha}{2K} \norm{ \sum_{k=1}^{K} \xdif^k }^2 \nonumber \\
    &\phantom{=}
    - \frac{\alpha}{2} \norm{ \frac{1}{K} \sum_{k=1}^{K} \pr{ \subg f(x^k) + u^k } }^2  - \frac{\alpha}{2} \norm{ \frac{1}{K} \sum_{k=1}^{K} \nabla h(p^k) }^2 + \inner{ x^0  +  \alpha u^0 }{ \frac{1}{K} \sum_{k=1}^{K} \xdif^k }. \label{eq:dys-gf-without-S2-1}
\end{align}
In the second equality, inner product terms with $x$ and $u$ are canceled out. As a detail for inner product terms with $x$, they are canceled out by \eqref{eq:dys-gf-conv-xdif} and the fact that
\begin{equation*}
        \inner{ x }{ \frac{1}{K} \sum_{k=1}^{K} \subg{f}(x^k) } + \frac{1}{K} \sum_{k=1}^{K} \inner{\nabla h(p^{k})}{x}  - \inner{x}{\frac{1}{K} \sum_{k=1}^{K} \xdif^k}  +  \langle x, \ubar^K \rangle =0 .
\end{equation*}
Again, using \eqref{eq:dys-gf-conv-xdif}, the first two terms on the left-hand side of \eqref{eq:dys-gf-without-S2-1} becomes
\begin{align}
    \MoveEqLeft[0.2] - \frac{1}{K} \sum_{k=1}^{K} \inner{ \subg{f}(x^k) }{ x^k }
        - \frac{1}{K} \sum_{k=1}^{K} \inner{u^k + \nabla h(p^k)}{x^k + \alpha \xdif^k} \nonumber \\
    &= - \frac{1}{K} \sum_{k=1}^{K} \inner{ \subg{f}(x^k) + u^k + \nabla h(x^k) }{ x^k }
       -  \frac{\alpha}{K} \sum_{k=1}^{K} \inner{ \xdif^k}{  u^k + \nabla h(p^k)} \nonumber \\
    &= \frac{1}{K} \sum_{k=1}^{K} \inner{ \xdif^k }{ -x^k - \alpha\pr{ u^k + \nabla h(p^k) } }.
\end{align}
Substituting it back to \eqref{eq:dys-gf-without-S2-1} and reorganizing, we obtain
\begin{align}
    \MoveEqLeft[0.2] (\ineqs_f + \ineqs_g + \ineqs_h - \sos_h - \sos_1) - \lhs \nonumber \\
    &= \frac{1}{K} \sum_{k=1}^{K} \inner{ \xdif^k }{ x^0  +  \alpha u^0 - x^k - \alpha\pr{ u^k + \nabla h(p^k) } } + \frac{\alpha}{K} \sum_{k=1}^{K} \norm{  \nabla h(p^k) }^2 \nonumber \\  
    &\phantom{=}
    - \frac{\alpha}{2K^2} \norm{ \sum_{k=1}^{K} \pr{ \xdif^k - \nabla h(p^k) } }^2  
    - \frac{\alpha}{2K^2} \norm{ \sum_{k=1}^{K} \nabla h(p^k) }^2 
    - \frac{\alpha}{2K} \norm{ \sum_{k=1}^{K} \xdif^k }^2 . \label{eq:dys-gf-without-S2-2}
\end{align}
Now, we eliminate $x^k$ in the first term on the right-hand side of \eqref{eq:dys-gf-without-S2-2}. Applying~\eqref{eq:dys-gf-reform-x-1}
and \eqref{eq:dys-gf-eliminate-p} recursively gives
\begin{equation*}
    x^k = x^{k-1} - \alpha \pr{ u^{k} - u^{k-1} } - \alpha \xdif^k
    = \dots = x^0 + \alpha u^0 - \alpha u^k - \alpha \sum_{l=1}^{k} \xdif^l.
\end{equation*}
Substituting it back to \eqref{eq:dys-gf-without-S2-2} gives
\begin{align}
    \MoveEqLeft[0.2] (\ineqs_f + \ineqs_g + \ineqs_h - \sos_h - \sos_1) - \lhs \nonumber \\
    &= \frac{\alpha}{K} \sum_{k=1}^{K} \sum_{l=1}^{k} \inner{ \xdif^k }{  \xdif^l } 
        - \frac{\alpha}{K} \sum_{k=1}^{K} \inner{ \xdif^k }{  \nabla h(p^k) }  
        - \frac{\alpha}{2K^2} \norm{ \sum_{k=1}^{K} \pr{ \xdif^k - \nabla h(p^k) } }^2  -  \frac{\alpha}{2K} \norm{ \sum_{k=1}^{K} \xdif^k }^2 \nonumber \\
    &\phantom{=}
        + \frac{\alpha}{2K} \sum_{k=1}^{K} \norm{ \nabla h(p^k) }^2 
        + \frac{\alpha}{2K} \sum_{k=1}^{K} \norm{  \nabla h(p^k) }^2
        - \frac{\alpha}{2K^2} \norm{ \sum_{k=1}^{K} \nabla h(p^k) }^2 . \label{eq:dys-gf-without-S2-3}
\end{align}
Finally, observe that for any $\{a^k\}_{k \in \natint} \subset \reals^n$, we have
\begin{equation} \label{eq:dys-gf-aux}
    \sum_{k=1}^{K} \norm{  a^k }^2 - \frac{1}{K} \norm{ \sum_{k=1}^{K}  a^k }^2 
    = \frac{1}{K} \sum_{k=1}^{K} (K-1) \norm{ a^k }^2 - \frac{1}{K} \sum_{k=1}^{K} \sum_{l=1}^{k-1} 2\inner{  a^k }{ a^l } 
    = \frac{1}{K} \sum_{k=1}^{K} \sum_{l=1}^{k-1} \norm{ a^k - a^l }^2,
\end{equation}
and thus
\begin{equation*}
    \begin{aligned}
        &\frac{\alpha}{2K} \sum_{k=1}^{K} \norm{  \nabla h(p^k) }^2
    - \frac{\alpha}{2K^2} \norm{ \sum_{k=1}^{K} \nabla h(p^k) }^2  
        =  \frac{\alpha}{2K^2} \sum_{k=1}^{K} \sum_{l=1}^{k-1} \norm{ \nabla h(p^k) - \nabla h(p^l) }^2.
    \end{aligned}
\end{equation*}
Next, with 
\begin{equation*}
    \frac{\alpha}{K} \sum_{k=1}^{K} \sum_{l=1}^{k} \inner{ \xdif^k }{  \xdif^l } 
    -  \frac{\alpha}{2K} \norm{ \sum_{k=1}^{K} \xdif^k }^2
    = \frac{\alpha}{2K} \sum_{k=1}^{K}  \norm{ \xdif^k }^2,
\end{equation*}
the first five terms on the right-hand side of \eqref{eq:dys-gf-without-S2-3} become
\begin{equation*}
    \begin{aligned}
        \MoveEqLeft[0.2] \frac{\alpha}{K} \sum_{k=1}^{K} \sum_{l=1}^{k} \inner{ \xdif^k }{  \xdif^l } 
        - \frac{\alpha}{K} \sum_{k=1}^{K} \inner{ \xdif^k }{  \nabla h(p^k) }  
        - \frac{\alpha}{2K^2} \norm{ \sum_{k=1}^{K} \pr{ \xdif^k - \nabla h(p^k) } }^2  -  \frac{\alpha}{2K} \norm{ \sum_{k=1}^{K} \xdif^k }^2 \\
        &\phantom{=} + \frac{\alpha}{2K} \sum_{k=1}^{K} \norm{  \nabla h(p^k) }^2 \\
        &= \frac{\alpha}{2K} \sum_{k=1}^{K}  \norm{ \xdif^k }^2 - \frac{\alpha}{K} \sum_{k=1}^{K} \inner{ \xdif^k }{  \nabla h(p^k) } + \frac{\alpha}{2K} \sum_{k=1}^{K} \norm{  \nabla h(p^k) }^2 - \frac{\alpha}{2K^2} \norm{ \sum_{k=1}^{K} \pr{ \xdif^k - \nabla h(p^k) } }^2 \\
        &= \frac{\alpha}{2K} \sum_{k=1}^{K}  \norm{ \xdif^k - \nabla h(p^k) }^2 - \frac{\alpha}{2K^2} \norm{ \sum_{k=1}^{K} \pr{ \xdif^k - \nabla h(p^k) } }^2 \\
        &= \frac{\alpha}{2K^2}  \sum_{k=1}^{K} \sum_{l=1}^{k-1}   \norm{ \xdif^k - \nabla h(p^k) - (\xdif^l - \nabla h(p^l)) }^2 ,
    \end{aligned}
\end{equation*}
where the last equation follows from \eqref{eq:dys-gf-aux}. Finally, combining with \eqref{eq:dys-gf-without-S2-3} yields
\begin{equation*}
    \begin{aligned}
        &(\ineqs_f + \ineqs_g + \ineqs_h - \sos_h - \sos_1) - \lhs \\
        &= \frac{\alpha}{2K^2} \sum_{k=1}^{K} \sum_{l=1}^{k-1}  \pr{ \norm{ \xdif^k - \nabla h(p^k) - (\xdif^l - \nabla h(p^l)) }^2 + \norm{ \nabla h(p^k) - \nabla h(p^l) }^2   } 
        = \sos_2 .
    \end{aligned}
\end{equation*}
Therefore, $\lhs = \ineqs_f + \ineqs_g + \ineqs_h - \sos_h - \sos_1 - \sos_2$, and we conclude the desired result.
\end{proof}

An ergodic $\cO(1/K)$ rate of convergence for \eqref{eq:dys-gf} follows immediately from \cref{prop:dys-gf-conv}.
\begin{theorem}[Convergence of \eqref{eq:dys-gf}] \label{thm:dys-gf-conv}
    Suppose $f$, $g$, and $h$ are CCP and $h$ is $L$-smooth (with $L>0$). Suppose $\{(x^k,u^k)\}_{k \in \natint}$ is generated by \eqref{eq:dys-gf} with step size $\alpha = \frac{1}{L}$. Then, for all $k \in \natint_+$, the averaged iterates $(\xbar^K, \ubar^K)$ defined in \eqref{eq:ergodic} satisfy
    \begin{equation} \label{eq:dys-gf-conv-gap}
        \cL(\xbar^K,u) - \cL(x,\ubar^K) \le \frac{1}{\alpha K} \pr{ \|x^0 - x\|^2 + \alpha^2 \|u^0 - u\|^2 }
    \end{equation}
    for all $x \in \dom f$ and all $u \in \dom g^\ast$. Moreover, \eqref{eq:dys-gf-conv-gap} holds with equality if and only if $\ineqs_f, \ineqs_g, \ineqs_h, \sos_h, \sos_1$ and $\sos_2$ are all zero.
\end{theorem}
\begin{proof}
    $\ineqs_f$ and $\ineqs_g$ are nonpositive since $f$ and $g$ are convex. It then follows from the convexity and $L$-smoothness of $h$ that $\ineqs_h$ is nonpositive. Moreover, $\sos_h$, $\sos_1$ and $\sos_2$ are nonnegative as they are sum of squares. Hence, we conclude \eqref{eq:dys-gf-conv-gap} from \cref{prop:dys-gf-conv}. The second conclusion is an immediate consequence of \cref{prop:dys-gf-conv} and the facts $\ineqs_f, \ineqs_g, \ineqs_h \le 0$ and $\sos_h, \sos_1, \sos_2 \ge 0$.
\end{proof}
The presented rate for \eqref{eq:dys-gf} is slower than that for its special case \eqref{eq:drs-gf} (see \cref{thm:drs-gf-conv}). In the next subsection, we showcase a simple example for which \eqref{eq:dys-gf} converges slower than $\cO(1/(K+1))$ and examine how the smooth term $h$ slows down convergence.

\subsection{\texorpdfstring{DYS-$gf$}{DYS-gf} fails to achieve an \texorpdfstring{$\cO(1/(K+1))$}{O(1/(K+1)} rate}

One may notice the difference in the rate \eqref{eq:dys-gf-conv-gap} for \eqref{eq:dys-gf} and that \eqref{eq:drs-gf-conv-gap} for \eqref{eq:drs-gf}. It is natural to ask whether \eqref{eq:dys-gf} can attain the $\cO(1/(K+1))$ rate as its special case \eqref{eq:drs-gf}. This section provides a negative answer by presenting a simple example in which \eqref{eq:dys-gf} converges strictly more slowly. Although the example does not match the exact $\cO(1/K)$ rate established in \cref{thm:dys-gf-conv}, it effectively demonstrates that the rate of \eqref{eq:dys-gf} can indeed be worse than $\cO(1/(K+1))$.

\begin{theorem} [Bad example for \eqref{eq:dys-gf}] \label{thm:dys-gf-worst-case}
    Under the same setting as in \cref{thm:dys-gf-conv}, for any $K \in \natint_+$ and any $\alpha > 0$, there exist CCP functions $f$ and $g$, $\tfrac{1}{\alpha}$-smooth convex function $h$, and points $x^0,u^0,x,u \in \reals^n$ such that $\|x^0 - x\|^2 + \alpha^2 \|u^0 - u\|^2 = 1$ and
    \begin{equation*}
        \cL(\xbar^K,u) - \cL(x,\ubar^K) > \frac{1}{\alpha (K+1)} \pr{ \|x^0 - x\|^2 + \alpha^2 \|u^0 - u\|^2 }.
    \end{equation*}
\end{theorem}

\begin{proof}
Fix $K \in \natint_+$ and $\alpha > 0$. Again, define $x^0 = e_0 / \sqrt{2}$, $u^0 = x^0 / \alpha$, and $x = u = 0$. (Recall $e_0 \in \reals^n$ denotes an arbitrary unit vector, so the example holds for any dimension $n$.) It is straightforward to check that $\|x^0 - x\|^2 + \alpha^2 \|u^0 - u\|^2 = 1$. We also denote
\[
    \eta_K := \frac{\sqrt{2} (K-1)}{K^2}.
\]
Consider
\[
    f(x) = \frac{(K-1) \eta_K}{\alpha} \|x\|, 
    \qquad g^\ast (y) = \eta_K \|y\|, 
    \qquad h(x) =  \frac{1}{2\alpha} \|x\|^2 .
\]
Thus, we have
\begin{equation*}
    \begin{aligned}
        \prox_{\alpha f} (x) &= \begin{cases}
            \pr{ \|{x}\| - (K-1)\eta_K } \frac{x}{\|{x}\|} \quad &\text{if} \ \|x\| \ge (K-1) \eta_K \\
            0 & \text{otherwise} 
        \end{cases} \\[2pt]
        \prox_{\alpha^{-1} g^\ast} (y) &= \begin{cases}
            \pr{ \|{y}\| - \frac{\eta_K}{\alpha} } \frac{y}{\|{y}\|} \quad &\text{if} \ \|y\| \ge \frac{\eta_K}{\alpha} \\
            0 & \text{otherwise.}
        \end{cases} 
    \end{aligned}
\end{equation*}
Then, we show by induction that
\begin{equation} \label{eq:dys-gf-ex-iter}
    x^k = \begin{cases}
        -(\sqrt{2} - K \eta_K) e_0 \quad &\text{if} \ k = 1 \\ 
        0 &\text{if} \ k = 2,3,\ldots,K,
    \end{cases} \quad \;
    u^k = \begin{cases}
        \frac{1}{\alpha}(\sqrt{2} - \eta_K) e_0 \quad &\text{if} \ k = 1 \\[2pt]
        \frac{1}{\alpha} (K-k) \eta_K e_0 &\text{if} \ k = 2,3,\ldots,K.
    \end{cases}
\end{equation}
\begin{enumerate}[label=(\roman*)]
    \item When $k=1$, recalling the definition \eqref{eq:dys-gf} and \eqref{eq:dys-gf-p}, 
    it is straightforward to verify that
    \begin{equation*}
    \begin{aligned}
        u^1 &= \prox_{\alpha^{-1} g^\ast} (u^0 + \tfrac{1}{\alpha} x^0) = \prox_{\alpha^{-1} g^\ast} (\tfrac{\sqrt{2}}{\alpha} e_0) = \tfrac{1}{\alpha} \big( \sqrt{2} - \eta_K \big) e_0 \\
        p^1 &= \alpha(u^0 + \tfrac{1}{\alpha} x^0) - \alpha u^1 = \eta_K e_0 \\
        x^1 &= \prox_{\alpha f} (p^{1} - \alpha \nabla h(p^{1}) - \alpha u^{1}) = \prox_{\alpha f} (- (\sqrt{2} - \eta_K) e_0) = -(\sqrt{2} - K \eta_K) e_0,
    \end{aligned}
    \end{equation*}
    where we use the fact $\nabla h(x) = \tfrac{1}{\alpha} x$ for all $x \in \reals^n$, $-\frac{-\alpha u^1}{\|-\alpha u^1\|} = e_0$, and
    \[
        \sqrt{2} - \eta_K
        = \pr{ 1 - \frac{K-1}{K^2} } \sqrt{2}
        = \frac{K^2-K+1}{K^2} \sqrt{2}
        > \frac{(K-1)^2}{K^2} \sqrt{2} = (K-1) \eta_K.
    \]
    Similarly, when $k=2$, we have
    \begin{equation*}
    \begin{aligned}
        u^2 &= \prox_{\alpha^{-1} g^\ast} (u^1 + \tfrac{1}{\alpha} x^1) = \prox_{\alpha^{-1} g^\ast} (\tfrac{1}{\alpha} (K-1) \eta_K e_0) = \tfrac{1}{\alpha} (K-2) \eta_K e_0  \\
        p^2 &= \alpha(u^1 + \tfrac{1}{\alpha} x^1) - \alpha u^2 = \eta_K  e_0 \\
        x^2 &= \prox_{\alpha f} (p^{2} - \alpha \nabla h(p^{2}) - \alpha u^{2}) = \prox_{\alpha f}(- \alpha u^{2}) = 0,
    \end{aligned}
    \end{equation*}
    where the last equality follows from the fact $\|-\alpha u^2\| = (K-2) \eta_K < (K-1) \eta_K$.

    \item Assume that the induction hypothesis is true for $k = m \leq K-1$. Then, by the induction hypothesis, we have
    \[
        \|u^m + \tfrac{1}{\alpha} x^m\| \geq \tfrac{\eta_K}{\alpha}, \qquad \|-\alpha u^m\| \leq (K-1) \eta_K.
    \]
    Thus, we have $x^{m+1} = 0$ and
    \[
        u^{m+1} = \prox_{\alpha^{-1} g^\ast} (u^m + \tfrac{1}{\alpha} x^m)
        = \tfrac{1}{\alpha} (K-m-1) \eta_K e_0.
    \]
    So the expression \eqref{eq:dys-gf-ex-iter} holds for $k = m+1 \leq K$.
\end{enumerate}
From \eqref{eq:dys-gf-ex-iter}, the ergodic sequence can be written as
\begin{equation} \label{eq:dys-gf-ex-ergodic}
    \begin{aligned}
        \xbar^K &= \frac{1}{K} x^1 = -  \frac{1}{K} \pr{ \sqrt{2} - K \eta_K } e_0  = \frac{1}{K} \pr{ 1 - \frac{K-1}{K} } \sqrt{2} e_0 = -\frac{\sqrt{2}}{K^2} e_0 , \\
        \ubar^K &= \frac{1}{\alpha K} \pr{ \pr{ \sqrt{2} -\eta_K } +  \sum_{k=2}^{K} \eta_K (K-k) } e_0
        = \frac{\sqrt{2} (K^2 - 2K + 3)}{2 \alpha K^2} e_0,
    \end{aligned}
\end{equation}
where the last equality follows from
\begin{align*}
    \sqrt{2} - \eta_K +  \sum_{k=2}^{K} (K-k) \eta_K 
    &= \sqrt{2} + \pr{ - 1 + \sum_{k=1}^{K-2} k } \eta_K
    = \sqrt{2} + \pr{ - 1 + \frac{(K-1)(K-2)}{2}  } \frac{K-1}{K^2} \sqrt{2} \\
    &= \frac{ ( K - 3 )(K-1)}{2K} \sqrt{2} +  \sqrt{2} 
    = \frac{\sqrt{2} (K^2 - 2K + 3)}{2K}.
\end{align*}
Finally, substituting \eqref{eq:dys-gf-ex-ergodic} and $x=u=0$ into the performance metric \eqref{eq:gap-func} yields 
\begin{align*}
    \cL(\bar{x}^K, u) - \cL(x, \bar{u}^K) &= f(\bar{x}^K) + h(\bar{x}^K) + \langle u, \bar{x}^K \rangle - g^\ast(u) - (f(x) + h(x) + \langle \bar{u}^K, x \rangle - g^\ast(\bar{u}^K)) \\
    &= f(\bar{x}^K) + h(\bar{x}^K) + g^\ast(\bar{u}^K) \\
    &= \frac{(K-1) \eta_K}{\alpha} \norm{ -\frac{1}{K^2} \sqrt{2} e_0  } + \frac{1}{2\alpha} \norm{-\frac{\sqrt{2}}{K^2}  e_0  }^2 + \eta_K \norm{ \frac{\sqrt{2}}{2 \alpha K^2} \pr{ K^2 - 2K + 3 } e_0 }   \\
    &= \frac{1}{\alpha} \frac{(K-1)^2}{K^2} \sqrt{2} \pr{ \frac{\sqrt{2}}{K^2}   } + \frac{1}{2\alpha} \pr{\frac{\sqrt{2}}{K^2}   }^2 +  \frac{K-1}{K^2} \frac{1}{\alpha K^2} \pr{ K^2 - 2K + 3 }    \\
    &= \frac{1}{\alpha K^4} \pr{ 2(K^2 - 2K + 1) + 1 + (K-1) \pr{ K^2 - 2K + 3 } }   \\
    &= \frac{1}{\alpha K^4} \pr{ 2K^2 - 4K + 2 + 1 + K^3 - 3K^2 + 5K - 3 }   
    = \frac{K^2-K+1}{ \alpha K^3} .
\end{align*}
Finally, the desired result follows from
\begin{equation*}
    \frac{K^2-K+1}{K^3} > \frac{1}{K+1}
    \quad \iff \quad
    K^3+1 = (K+1)(K^2-K+1) > K^3.  \qedhere
\end{equation*}
\end{proof}
Note that the rate $\frac{K^2-K+1}{\alpha K^3}$ equals $\frac{1}{\alpha K}$ when $K=1$; that is, the bound in \cref{thm:dys-gf-conv} is tight at $K=1$. We conjecture that the rate in \cref{thm:dys-gf-conv} is tight in general, but leave a formal proof as an open question for future work. 

\subsection{\texorpdfstring{DYS-$fg$}{DYS-fg}: restoring the \texorpdfstring{$\cO(1/(K+1))$}{O(1/(K+1))} rate by swapping \texorpdfstring{$f$}{f} and \texorpdfstring{$g$}{g}}

We proceed to analyze \eqref{eq:dys-fg}. As before, we reformulate \eqref{eq:dys-fg} as
\begin{subequations} \label{eq:dys-fg-reform}
    \begin{align}
        x^{k} &= x^{k-1} - \alpha (u^{k-1} + \nabla h(x^{k-1})) - \alpha \subg f(x^k) \label{eq:dys-fg-reform-x} \\
        p^{k} &= 
        x^k + (x^k - x^{k-1}) + \alpha ((\nabla h(x^{k-1}) + u^{k-1} ) - (\nabla h(x^{k}) + u^k ) ) \in \partial g^*(u^{k}). \label{eq:dys-fg-reform-p}
    \end{align}
\end{subequations}
We now establish the tight convergence rate of \eqref{eq:dys-fg}. As in the case of \eqref{eq:drs-fg}, we prove an equality that immediately yields the convergence of \eqref{eq:dys-fg}. Although \cref{prop:dys-fg-conv} reduces to \cref{prop:drs-fg-conv} when $h=0$, its proof is more involved because it must account for the additional smooth function $h$. 
\begin{proposition} \label{prop:dys-fg-conv}
    Suppose $f$, $g$, and $h$ are CCP and $h$ is $L$-smooth (with $L>0$). Suppose also that $\{(x^k,u^k)\}_{k \in \natint}$ is generated by \eqref{eq:dys-fg} with step size $\alpha = \frac{1}{L}$. Denote $\xbar^K, \ubar^K, \pbar^K, p$ as in \Cref{prop:dys-gf-conv}, and $p^k$ as in \eqref{eq:dys-fg-reform-p}. Then, for all $K \in \natint$ and all $x,u \in \reals^n$, the following equality holds
    \begin{equation} \label{eq:dys-fg-conv-eq}
        \cL(\xbar^K,u) - \cL(x,\ubar^K) - \frac{1}{\alpha (K+1)} \pr{ \|x^0-x\|^2 + \alpha^2 \|u^0 - u\|^2 } = \ineqs_f+\ineqs_g+\ineqs_h-\sos_1-\sos_2-\sos_h,
    \end{equation}
    where $(\xdif^k, \ineqs_f, \ineqs_g, \sos_2)$ are defined in \Cref{prop:dys-gf-conv}, and
    \begin{align*}
        \ineqs_f 
            &:= \resizebox{.93\textwidth}{!}{$\displaystyle
            \frac{1}{K} \sum_{k=1}^{K} \pr{ f ( \xbar^K ) - f(x^k) + \inner{\subg f ( \xbar^K )}{x^k - \xbar^K} } 
            + \frac{1}{K} \sum_{k=1}^{K}  \pr{ f(x^{k}) - f(x) + \inner{\subg f(x^{k})}{x - x^{k}} }$} \\
        \ineqs_g 
            &:= \frac{1}{K} \sum_{k=1}^{K} \pr{ g(p^k) - g( \pbar^K ) + \inner{u^k}{\pbar^K - p^k} } 
             +  \frac{1}{K} \sum_{k=1}^{K}  \pr{ g(p) - g(p^{k}) + \inner{u}{p^{k}-p} }  \\
        \ineqs_h 
            &:= \frac{1}{K} \sum_{k=1}^{K} \pr{ h(\xbar^K) - h(x^k) + \inner{\nabla h(\xbar^K)}{x^k - \xbar^K} + \frac{\alpha}{2} \norm{ \nabla h(x^k) - \nabla h(\xbar^K) }^2 } \\
            &\phantom{=} + \frac{K-1}{K(K+1)} \sum_{k=1}^{K}  \pr{ h(x^{k}) - h(x) + \inner{\nabla h(x^{k})}{x - x^{k}} + \frac{\alpha}{2} \norm{ \nabla h(x^k) -  \nabla h(x) }^2} \\
            &\phantom{=}+ \frac{2}{K+1} \pr{ h(x^{0}) - h(x) + \inner{\nabla h(x^{0})}{x - x^{0}} + \frac{\alpha}{2} \norm{ \nabla h(x^0) -  \nabla h\pr{ x } }^2} \\
            &\phantom{=} + \frac{2}{K(K+1)} \sum_{k=1}^{K}  \pr{ h(x^{k}) - h(x^0) + \inner{\nabla h(x^{k})}{x^0 - x^{k}} + \frac{\alpha}{2} \norm{ \nabla h(x^k) - \nabla h(x^0) }^2} \\
        \sos_h 
            &:= \frac{\alpha}{2} \norm{ \nabla h(\bar{x}^K) - \frac{1}{K} \sum_{k=1}^{K} \nabla h(x^k) }^2 + \frac{\alpha(K-1)}{2(K+1)} \norm{ \nabla h(x) - \frac{1}{K} \sum_{k=1}^{K} \nabla h(x^k) }^2    \\
        \sos_1 
            &:= \frac{1}{K+1} \Bigg( \frac{1}{ \alpha} \norm{ x^0 - x - \frac{\alpha (K+1)}{2K} \sum_{k=1}^{K} \xdif^k - \alpha \pr{ \nabla h(x^0) - \frac{1}{K} \sum_{k=1}^{K} \nabla h(x^k) } }^2 \\ 
            &\quad\qquad\quad            
            + \alpha \norm{ u^0 - u + \frac{K+1}{2K} \sum_{k=1}^{K} \pr{ \subg f(x^k) + u^k  + \nabla h(x^k)} }^2
            + \frac{\alpha}{K+1} \norm{ \nabla h(x^0) - \nabla h(x) }^2  \Bigg).
    \end{align*}
\end{proposition}
\begin{proof}
    We repeat some arguments done in \Cref{prop:drs-fg-conv}. First, we verify that $p^k$ in \eqref{eq:dys-fg-reform-p} can still be written as in \eqref{eq:drs-fg-eliminate-p}. From \eqref{eq:dys-fg-reform-x} we obtain $x^{k} - x^{k-1} = - \alpha (u^{k-1} + \nabla h(x^{k-1})) - \alpha \subg f(x^{k})$. Substituting it into \eqref{eq:dys-fg-reform-p} yields
    \begin{equation*}
        p^k = x^k - \alpha (\subg f(x^k) + u^k +  \nabla h(x^{k})) = x^k - \alpha \xdif^k. 
    \end{equation*}
    With the same argument of \Cref{prop:drs-fg-conv}, $\ineqs_f$ and $\ineqs_g$ simplify to
    \begin{equation*}
        \begin{aligned}
            \ineqs_f 
            &= f(\xbar^K) - f(x) + \inner{ x }{ \frac{1}{K} \sum_{k=1}^{K} \subg{f}(x^k) } - \frac{1}{K} \sum_{k=1}^{K} \inner{ \subg{f}(x^k) }{ x^k } \\
            \ineqs_g 
            &= g(p) - g(\pbar^K) - \inner{u}{p} + \inner{\ubar^K}{\pbar^{K}}  + \inner{u}{\xbar^K} - \alpha \inner{u}{ \frac{1}{K} \sum_{k=1}^{K} \xdif^k } - \frac{1}{K} \sum_{k=1}^{K} \inner{u^k}{x^k - \alpha \xdif^k},
        \end{aligned}
    \end{equation*}
    which is restated here for later reference.
    
    Now, to regroup some of the terms, we define
    \begin{align*}
        \tildeineqs_h 
             &= \ineqs_h -\frac{\alpha}{2K} \sum_{k=1}^{K} \norm{ \nabla h(\xbar^K) -  \nabla h(x^k) }^2 - \frac{\alpha(K-1)}{2K(K+1)} \sum_{k=1}^{K} \norm{  \nabla h(x) - \nabla h(x^k) }^2  \\ &\phantom{=} \quad\,\,
             - \frac{\alpha}{K+1} \norm{ \nabla h(x^0) -  \nabla h(x) }^2 - \frac{\alpha}{K(K+1)} \sum_{k=1}^{K}  \norm{  \nabla h(x^0) - \nabla h(x^k) }^2 \\
        \tildesos_h 
             &= \sos_h -\frac{\alpha}{2K} \sum_{k=1}^{K} \norm{ \nabla h(\xbar^K) -  \nabla h(x^k) }^2 -  \frac{\alpha(K-1)}{2K(K+1)} \sum_{k=1}^{K} \norm{  \nabla h(x) - \nabla h(x^k) }^2 \\
        \tildesos_1 
            &= \sos_1 - \frac{\alpha}{K+1} \norm{ \nabla h(x^0) -  \nabla h(x) }^2 - \frac{\alpha}{K(K+1)} \sum_{k=1}^{K}  \norm{  \nabla h(x^0) - \nabla h(x^k) }^2 .
    \end{align*}
    Moreover, the terms $h(x^k)$, $h(x^0)$, and $\langle \nabla h(x^k), {x^k-\xbar^k} \rangle$ cancel out, and thus $\tildeineqs_h$ simplifies to
    \begin{equation*}
        \begin{aligned}
            \tildeineqs_h 
            &= h(\xbar^K) - h(x) + \frac{2}{K(K+1)} \inner{ x^0 }{ \sum_{k=1}^{K} \nabla h(x^k) } 
         - \frac{1}{K} \sum_{k=1}^{K} \inner{\nabla h(x^{k})}{x^{k}} \\ &\quad
         + \frac{K-1}{K(K+1)} \sum_{k=1}^{K} \inner{\nabla h(x^{k})}{x }
         + \frac{2}{K+1} \inner{\nabla h(x^0)}{x-x^0}, 
        \end{aligned}
    \end{equation*}
    where first two terms appears in the left-hand side of \eqref{eq:dys-gap}.

    For $\tildesos_h$, straightforward calculations show that
    \begin{align*}
        \tildesos_h 
            &= \frac{\alpha}{2} \norm{ \nabla h(\bar{x}^K) - \frac{1}{K} \sum_{k=1}^{K} \nabla h(x^k) }^2 
            - \frac{\alpha}{2K} \sum_{k=1}^{K} \norm{ \nabla h(\xbar^K) -  \nabla h(x^k) }^2 \\
        &\phantom{=}
            + \frac{\alpha(K-1)}{2(K+1)} \norm{ \nabla h(x) - \frac{1}{K} \sum_{k=1}^{K} \nabla h(x^k) }^2 
            - \frac{\alpha(K-1)}{2K(K+1)} \sum_{k=1}^{K} \norm{  \nabla h(x) - \nabla h(x^k) }^2 \\
        &= - \frac{\alpha}{2K} \sum_{k=1}^K \norm{ \nabla h(x^k) }^2 + \frac{\alpha}{2} \norm{ \frac{1}{K} \sum_{k=1}^{K} \nabla h(x^k) }^2 \\
        &\phantom{=}
            + \frac{\alpha(K-1)}{2(K+1)} \norm{ \nabla h(x) - \frac{1}{K} \sum_{k=1}^{K} \nabla h(x^k) }^2 
            - \frac{\alpha(K-1)}{2K(K+1)} \sum_{k=1}^{K} \norm{  \nabla h(x) - \nabla h(x^k) }^2 \\
        &= - \frac{\alpha}{K+1} \sum_{k=1}^K \norm{ \nabla h(x^k) }^2 + \frac{\alpha K}{K+1} \norm{ \frac{1}{K} \sum_{k=1}^{K} \nabla h(x^k) }^2.
    \end{align*}
    The first equality follows from the definition of $\tildesos_h$. In the second equality, we cancel out all the $\nabla h(\xbar^K)$ terms, and the last equality cancels out the $\nabla h(x)$ terms.
       
    Next, we move on to $\tildesos_1$. Observe that the sum of the first and the last term of $\tildesos_1$ can be rewritten as:
    % \small
    \begin{align*}
        \MoveEqLeft[0.2] \frac{1}{ \alpha(K+1)} \norm{ x^0 - x - \frac{\alpha (K+1)}{2K} \sum_{k=1}^{K} \xdif^k - \alpha \pr{ \nabla h(x^0) - \frac{1}{K} \sum_{k=1}^{K} \nabla h(x^k) } }^2 \\
        &\phantom{=} 
        - \frac{\alpha}{K(K+1)} \sum_{k=1}^{K}  \norm{  \nabla h(x^0) - \nabla h(x^k) }^2 \\
        &= 
        \frac{1}{ \alpha(K+1)} \norm{ x^0 - x   - \alpha \frac{K+1}{2K} \sum_{k=1}^{K} \xdif^k }^2
        - \frac{2}{ K+1} \inner{ x^0 - x }{ \nabla h(x^0) } \\
        &\phantom{=} 
        + \frac{2}{ K(K+1) } \inner{ x^0 - x }{ \sum_{k=1}^{K} \nabla h(x^k) } 
        + \inner{ \frac{1}{ K }  \sum_{k=1}^{K} \xdif^k }{ \alpha\nabla h(x^0) } - \frac{\alpha}{K^2} \inner{ \sum_{k=1}^{K} \xdif^k }{ \sum_{k=1}^{K} \nabla h(x^k)} \\
        &\phantom{=}
        - \frac{\alpha}{K(K+1)} \sum_{k=1}^K \norm{ \nabla h(x^k) }^2 + \frac{\alpha}{K+1} \norm{ \frac{1}{K} \sum_{k=1}^{K} \nabla h(x^k) }^2 \\ 
        &=
        \frac{1}{ \alpha(K+1)} \norm{ x^0 - x }^2 
        - \inner{x^0 - \alpha \nabla h(x^0) - x}{ \frac{1}{K}  \sum_{k=1}^{K} \xdif^k } 
        + \frac{\alpha(K+1)}{4K^2} \norm{  \sum_{k=1}^{K} \xdif^k }^2 \\
        &\phantom{=}
        - \frac{2}{K+1} \inner{ x^0 - x }{ \nabla h(x^0) }
        + \frac{2}{ K(K+1) } \inner{ x^0 - x }{ \sum_{k=1}^{K} \nabla h(x^k) } 
        - \frac{\alpha}{K^2} \sum_{k=1}^{K} \sum_{l=1}^{K} \inner{ \xdif^l }{ \nabla h(x^k)} \\
        &\phantom{=}
        - \frac{\alpha}{K(K+1)} \sum_{k=1}^K \norm{ \nabla h(x^k) }^2 
        + \frac{\alpha}{K+1} \norm{ \frac{1}{K} \sum_{k=1}^{K} \nabla h(x^k) }^2.
    \end{align*}
    \normalsize
    In the first equality, we expand the two squared terms and distribute the cross terms. In the second equality, we expand the first square and gather the inner product terms with $\sum_{k=1}^{K} \xdif^k$. 
    Similarly, the second term in $\tildesos_1$ can be rewritten as 
    \begin{align*}
        \MoveEqLeft[0.2] \frac{\alpha}{K+1} \norm{ u^0 - u + \frac{K+1}{2K} \sum_{k=1}^{K} \xdif^k }^2 \\
        &= \frac{\alpha}{K+1} \norm{ u^0 - u }^2 + {\alpha} \inner{ u^0 - u}{ \frac{1}{K} \sum_{k=1}^{K} \xdif^k } + \frac{\alpha (K+1)}{4K^2} \norm{\sum_{k=1}^{K}\xdif^k}^2.
    \end{align*}
    
    Now, gathering the observations, we denote by \lhs\ the left-hand side of \eqref{eq:dys-fg-conv-eq} and obtain
    \begin{align}
        \MoveEqLeft[0.2] (\ineqs_f + \ineqs_g + \ineqs_h - \sos_h - \sos_1) - \lhs \nonumber \\
        &= (\tildeineqs_f + \tildeineqs_g + \tildeineqs_h - \tildesos_h - \tildesos_1) - \lhs \nonumber \\
        &= 
        \inner{ x }{ \frac{1}{K} \sum_{k=1}^{K} \subg{f}(x^k) } - \frac{1}{K} \sum_{k=1}^{K} \inner{ \subg{f}(x^k) }{ x^k } 
        - \alpha \inner{u}{ \frac{1}{K} \sum_{k=1}^{K} \xdif^k } - \frac{1}{K} \sum_{k=1}^{K} \inner{u^k}{x^k - \alpha \xdif^k} \nonumber \\ 
        &\phantom{=}
        + \frac{2}{K(K+1)} \inner{ x^0 }{ \sum_{k=1}^{K} \nabla h(x^k) } + \frac{K-1}{K(K+1)} \sum_{k=1}^{K} \inner{\nabla h(x^{k})}{x }
        - \frac{1}{K} \sum_{k=1}^{K} \inner{\nabla h(x^{k})}{x^{k}} \nonumber \\
        &\phantom{=} 
        + \frac{\alpha}{K} \sum_{k=1}^K \norm{ \nabla h(x^k) }^2 
        - \alpha \norm{ \frac{1}{K} \sum_{k=1}^{K} \nabla h(x^k) }^2 + \inner{ x^0 - \alpha \pr{u^0 + \nabla h(x^0)} }{ \frac{1}{K} \sum_{k=1}^{K} \xdif^k } \nonumber \\
        &\phantom{=}
        - \inner{ x - \alpha u }{ \frac{1}{K} \sum_{k=1}^{K} \xdif^k } 
        - \frac{\alpha (K+1)}{2K^2} \norm{ \sum_{k=1}^{K} \xdif^k }^2 
        - \frac{2}{ K(K+1) } \inner{ x^0 - x }{ \sum_{k=1}^{K} \nabla h(x^k) } \nonumber \\
        &\phantom{=}
        + \frac{\alpha}{K^2} \sum_{k=1}^{K} \sum_{l=1}^{K} \inner{  \nabla h(x^k)}{  \xdif^l } + \langle x, \ubar^K \rangle \nonumber \\
        &= 
        - \frac{1}{K} \sum_{k=1}^{K} \inner{ \subg{f}(x^k) }{ x^k }
        - \frac{1}{K} \sum_{k=1}^{K} \inner{u^k}{x^k - \alpha \xdif^k} 
        - \frac{1}{K} \sum_{k=1}^{K} \inner{\nabla h(x^{k})}{ x^{k}} \nonumber \\ 
        &\phantom{=}
        + \frac{\alpha}{K} \sum_{k=1}^{K} \norm{  \nabla h(x^k) }^2
        - \alpha \norm{ \frac{1}{K} \sum_{k=1}^{K} \nabla h(x^k) }^2 + \inner{ x^0 - \alpha \pr{u^0 + \nabla h(x^0)} }{ \frac{1}{K} \sum_{k=1}^{K} \xdif^k } \nonumber \\
        &\phantom{=}
        - \frac{\alpha(K+1)}{2K^2} \norm{ \sum_{k=1}^{K} \xdif^k }^2 + \frac{\alpha}{K^2} \sum_{k=1}^{K} \sum_{l=1}^{K} \inner{  \nabla h(x^k)}{ \xdif^l }, \label{eq:dys-fg-without-S2-1}
    \end{align}
    In the second equality, the terms $\frac{2}{K+1} \langle{\nabla h(x^0)}, {x-x^0} \rangle$, $\frac{2}{K(K+1)} \langle{ x^0 }, { \sum_{k=1}^{K} \nabla h(x^k) } \rangle$ and the inner product terms with $x$ and $u$ are canceled out. 
    As a detail for the inner product terms with $x$, they are canceled out since
    \begin{align*}
        0 &= \inner{ x }{ \frac{1}{K} \sum_{k=1}^{K} \subg{f}(x^k) } + \frac{K-1}{K(K+1)} \sum_{k=1}^{K} \inner{\nabla h(x^{k})}{x} - \inner{x}{\frac{1}{K} \sum_{k=1}^{K} \xdif^k} \\
        &\phantom{=}
        + \frac{2}{ K(K+1) } \inner{ x }{ \sum_{k=1}^{K} \nabla h(x^k) } + \inner{x}{\ubar^K}, 
    \end{align*}
    where we also use the definition of $v^k$ \eqref{eq:dys-gf-conv-xdif}. Again, it follows from \eqref{eq:dys-gf-conv-xdif} that the first three terms on the right-hand side of \eqref{eq:dys-fg-without-S2-1} is 
    \begin{align}
        \MoveEqLeft[0.2] - \frac{1}{K} \sum_{k=1}^{K} \inner{ \subg{f}(x^k) }{ x^k }
            - \frac{1}{K} \sum_{k=1}^{K} \inner{u^k}{x^k - \alpha \xdif^k} 
            - \frac{1}{K} \sum_{k=1}^{K} \inner{\nabla h(x^{k})}{ x^{k}} \nonumber \\
        &= - \frac{1}{K} \sum_{k=1}^{K} \inner{ \subg{f}(x^k) + u^k + \nabla h(x^k) }{ x^k }
            +  \frac{1}{K} \sum_{k=1}^{K} \inner{ \xdif^k }{ \alpha u^k }
            = \frac{1}{K} \sum_{k=1}^{K} \inner{ \xdif^k }{ -x^k + \alpha u^k }. \label{eq:dys-fg-without-S2-2}
    \end{align}
    Substituting \eqref{eq:dys-fg-without-S2-2} back to \eqref{eq:dys-fg-without-S2-1} gives
    \begin{align}
        \MoveEqLeft[0.2] (\ineqs_f + \ineqs_g + \ineqs_h - \sos_h - \sos_1) - \lhs \nonumber \\
        &= \frac{1}{K} \sum_{k=1}^{K} \inner{ \xdif^k }{ x^0 - \alpha \pr{u^0 + \nabla h(x^0)} - x^k + \alpha u^k } 
        + \frac{\alpha}{K} \sum_{k=1}^{K} \norm{  \nabla h(x^k) }^2 \nonumber \\
        &\phantom{=}
        - \frac{\alpha}{K^2} \norm{ \sum_{k=1}^{K} \nabla h(x^k) }^2 -  \frac{\alpha(K+1)}{2K^2} \norm{ \sum_{k=1}^{K} \xdif^k }^2
        + \frac{\alpha}{K^2} \sum_{k=1}^{K} \sum_{l=1}^{K} \inner{  \nabla h(x^k)}{  \xdif^l }. \label{eq:dys-fg-without-S2-3}
    \end{align}
    
    Now, we eliminate $x^k$. Applying \eqref{eq:dys-fg-reform-x} recursively, we obtain
    \begin{equation*}
        x^k
        = \dots
        = x^0 - \alpha \sum_{l=0}^{k-1} \pr{ u^{l} + \nabla h(x^{l-1}) + \subg{f}(x^{l+1}) } 
        = x^0 - \alpha \pr{u^0 + \nabla h(x^0)} - \alpha \subg{f}(x^k) - \alpha \sum_{l=1}^{k-1} \xdif^l.
    \end{equation*}
    Substituting it back to \eqref{eq:dys-fg-without-S2-3} gives
    \begin{align}
        &(\ineqs_f + \ineqs_g + \ineqs_h - \sos_h - \sos_1) - \lhs \nonumber \\
        &= \frac{\alpha}{K} \sum_{k=1}^{K} \sum_{l=1}^{k} \inner{ \xdif^k }{  \xdif^l } 
        - \frac{\alpha(K+1)}{2K^2} \norm{ \sum_{k=1}^{K} \xdif^k }^2
        + \frac{\alpha}{K} \sum_{k=1}^{K} \norm{  \nabla h(x^k) }^2 \nonumber \\
        &\phantom{=}
        - \frac{\alpha}{K} \sum_{k=1}^{K} \inner{ \xdif^k }{  \nabla h(x^k) } 
        - \frac{\alpha}{K^2} \norm{ \sum_{k=1}^{K} \nabla h(x^k) }^2 
        + \frac{\alpha}{K^2} \sum_{k=1}^{K} \sum_{l=1}^{K} \inner{  \nabla h(x^k)}{  \xdif^l }. \label{eq:dys-fg-without-S2-4}
    \end{align}
    Finally, recalling \eqref{eq:drs-gf-prf-last} we know
    \begin{equation*}
        \frac{\alpha}{K} \sum_{k=1}^{K} \sum_{l=1}^{k} \inner{ \xdif^k }{  \xdif^l } 
            -  \frac{\alpha(K+1)}{2K^2} \norm{ \sum_{k=1}^{K} \xdif^k }^2
        = \frac{\alpha}{2K^2} \sum_{k=1}^{K} \sum_{l=1}^{k-1} \norm{ \xdif^k - \xdif^l }^2,
    \end{equation*}
    and observe that
    \begin{equation*}
        \begin{aligned}
        &\frac{\alpha}{K} \sum_{k=1}^{K} \norm{  \nabla h(x^k) }^2
        - \frac{\alpha}{K} \sum_{k=1}^{K} \inner{ \xdif^k }{  \nabla h(x^k) } - \frac{\alpha}{K^2} \norm{ \sum_{k=1}^{K} \nabla h(x^k) }^2
        + \frac{\alpha}{K^2} \sum_{k=1}^{K} \sum_{l=1}^{K} \inner{  \nabla h(x^k)}{  \xdif^l }  \\
        &= \frac{\alpha}{K^2} \pr{ \sum_{k=1}^{K} \sum_{l=1}^{K} \norm{  \nabla h(x^k) }^2 
        -\sum_{k=1}^{K} \sum_{l=1}^{K} \inner{  \nabla h(x^k) }{ \xdif^k } 
        - \sum_{k=1}^{K} \sum_{l=1}^{K} \inner{ \nabla h(x^k) }{ \nabla h(x^l) } 
        + \sum_{k=1}^{K} \sum_{l=1}^{K} \inner{  \nabla h(x^k)}{  \xdif^l } } \\
        &= \frac{\alpha}{K^2} \sum_{k=1}^{K} \sum_{l=1}^{K} \inner{\nabla h(x^k)}{\nabla h(x^k) - \xdif^k - \nabla h(x^l) + \xdif^l} \\
        &= \frac{\alpha}{2K^2} \sum_{k=1}^{K} \sum_{l=1}^{K} \inner{\nabla h(x^k)}{\nabla h(x^k) - \xdif^k - \nabla h(x^l) + \xdif^l} \\
        &\phantom{=}
        + \frac{\alpha}{2K^2} \sum_{l=1}^{K} \sum_{k=1}^{K}  \inner{\nabla h(x^l)}{\nabla h(x^l) - \xdif^l - \nabla h(x^k) +  \xdif^k} \\
        &= \frac{\alpha}{2K^2} \sum_{k=1}^{K} \sum_{l=1}^{K} \inner{\nabla h(x^k) - \nabla h(x^l)}{\nabla h(x^k) - \xdif^k - \nabla h(x^l) +  \xdif^l}  \\
        &= \frac{\alpha}{K^2} \sum_{k=1}^{K} \sum_{l=1}^{k-1} \inner{\nabla h(x^k) - \nabla h(x^l)}{\nabla h(x^k) - \xdif^k - \nabla h(x^l) +  \xdif^l}  .
        \end{aligned}
    \end{equation*}
    Combining with \eqref{eq:dys-fg-without-S2-4} gives
    \begin{equation*}
        \begin{aligned}
            &(\ineqs_f + \ineqs_g + \ineqs_h - \sos_h - \sos_1) - \lhs \\
            &= \frac{\alpha}{2K^2} \sum_{k=1}^{K} \sum_{l=1}^{k-1} \pr{ \norm{ \xdif^k - \xdif^l }^2 + 2\inner{\nabla h(x^k) - \nabla h(x^l)}{\nabla h(x^k) - \xdif^k - \nabla h(x^l) +  \xdif^l} } \\
            &= \frac{\alpha}{2K^2} \sum_{k=1}^{K} \sum_{l=1}^{k-1}  \pr{ \norm{ \xdif^k - \nabla h(x^k) - (\xdif^l - \nabla h(x^l)) }^2 + \norm{ \nabla h(x^k) - \nabla h(x^l) }^2   } 
            = \sos_2,
        \end{aligned}
    \end{equation*}
    which is the desired result.
\end{proof}

\begin{theorem}[Convergence of \eqref{eq:dys-fg}] \label{thm:dys-fg-conv}
    Suppose $f$, $g$, and $h$ are CCP and $h$ is $L$-smooth (with $L>0$). Suppose $\{(x^k,u^k)\}_{k \in \natint}$ is generated by \eqref{eq:dys-fg} with step size $\alpha = \frac{1}{L}$. Then, for all $K \in \natint_+$, the averaged iterates $(\xbar^K, \ubar^K)$ defined in \eqref{eq:ergodic} satisfy
    \begin{equation} \label{eq:dys-fg-conv-gap}
        \cL(\xbar^K,u) - \cL(x,\ubar^K) \le \frac{1}{\alpha (K+1)} \pr{ \|x^0 - x\|^2 + \alpha^2 \|u^0 - u\|^2 }
    \end{equation}
    for all $x \in \dom f$ and all $u \in \dom g^\ast$.
\end{theorem}
\begin{proof}
    It follows from the convexity of $f$ and $g$ that $\ineqs_f$ and $\ineqs_g$ are nonpositive and from the convexity and $L$-smoothness of $h$ that $\ineqs_h$ is nonpositive. Moreover, $\sos_h$, $\sos_1$ and $\sos_2$ are nonnegative since they are sum of squares.  Therefore, we conclude \eqref{eq:dys-fg-conv-gap} from \cref{prop:dys-fg-conv}.
\end{proof}
Recall that \eqref{eq:dys-fg} reduces to \eqref{eq:drs-fg} when $h = 0$. So, \eqref{eq:dys-fg} cannot exhibit a faster worst-case rate than \eqref{eq:drs-fg}. However, \cref{thm:dys-fg-conv} presents the same upper bound for \eqref{eq:dys-fg} as for \eqref{eq:drs-fg}. Combined with the tightness of our \eqref{eq:drs-fg} rate, we readily conclude that the $\cO(1/(K+1))$ ergodic rate in \cref{thm:dys-fg-conv} must be tight for \eqref{eq:dys-fg}. This argument is not evident prior to obtaining \cref{thm:dys-fg-conv}. So we present the tightness of \eqref{eq:dys-fg-conv-gap} as an corollary.
\begin{corollary}[Worst-case example for \eqref{eq:dys-fg}] \label{cor:dys-fg-worst-case}
    Under the same setting as in \cref{thm:dys-fg-conv}, for any $K \in \natint_+$ and any $\alpha > 0$, there exist \CCP\ functions $f$ and $g$, a $\tfrac{1}{\alpha}$-smooth convex function $h$, and points $x^0, u^0, x, u \in \reals^n$ such that $\|x^0 - x\|^2 + \alpha^2 \|u^0 - u\|^2 = 1$ and
    \begin{equation*} 
        \cL(\xbar^K,u) - \cL(x,\ubar^K) = \frac{1}{\alpha (K+1)} \pr{ \|x^0 - x\|^2 + \alpha^2 \|u^0 - u\|^2 }.
    \end{equation*}
\end{corollary}
\begin{proof}
    Note that \cref{prop:dys-fg-conv} reduces to \cref{prop:drs-fg-conv} when $h = 0$. Thus, following the same argument as in the proof of \cref{thm:drs-gf-worst-case}, we can show that the worst-case example for \eqref{eq:drs-fg} in \cref{thm:drs-gf-worst-case}, together with $h = 0$, also serves as a worst-case example for \eqref{eq:dys-fg}. 
\end{proof}

\subsection{Discussion: comparison between the two variants}

For \eqref{eq:dys-fg}, a worst-case example exists with $h=0$, and in fact, the same example serves both \eqref{eq:dys-fg} and its special case \eqref{eq:drs-fg}. This naturally arises the question of whether \eqref{eq:dys-gf} and \eqref{eq:drs-gf} can also share the same worst-case example, or whether a worst-case instance for \eqref{eq:dys-gf} can be constructed with $h=0$. The following proposition provides a negative answer: \eqref{eq:dys-gf} restores the $\cO(1/(K+1))$ rate as long as either $g=0$ or $h=0$.
\begin{proposition} \label{prop:dys-h=0}
    Suppose $\{(x^k,u^k)\}_{k \in \natint}$ is generated by \eqref{eq:dys-gf} with step size $\alpha > 0$. Denote $\xbar^K$, $\ubar^K$ as in \eqref{eq:ergodic}. For any integer $K \geq 1$, there do \emph{not} exist CCP functions $f$ and $g$, $\tfrac{1}{\alpha}$-smooth convex function $h$, and points $x^0,u^0,x,u \in \reals^n$ such that at least one of $g$ or $h$ vanishes, $\|x^0-x\|^2 + \alpha^2 \|u^0 - u\|^2 = 1$, and
    \begin{equation*}
        \cL(\xbar^K,u) - \cL(x,\ubar^K) > \frac{1}{\alpha (K+1)} \pr{ \|x^0 - x\|^2 + \alpha^2 \|u^0 - u\|^2 }.
    \end{equation*}
\end{proposition}
\begin{proof} 
    When $g=0$, $g^\ast = \delta_{\set{0}}$ and thus $\prox_{\alpha^{-1} g^\ast} (y) = 0$ for all $y \in \reals^n$. Then, the sequence $\{x^k\}_{k \in \natint}$ generated by \eqref{eq:dys-gf} reduces to a special case of \eqref{eq:dys-fg} with $g=0$. So, the desired conclusion follows from \cref{thm:dys-fg-conv}.

    When $h=0$, the sequence $\{(x^k,u^k)\}_{k \in \natint}$ generated by \eqref{eq:dys-gf} reduces to \eqref{eq:drs-gf}. Therefore we obtain the desired conclusion from \cref{thm:drs-gf-conv}.     
\end{proof}

\Cref{prop:dys-h=0} reveals that, in the worst case, the presence of the smooth term $h$ slows down the convergence of \eqref{eq:dys-gf}. In contrast, this phenomenon does not occur in \eqref{eq:dys-fg}. This asymmetry highlights a key distinction between the two variants of DYS: while \eqref{eq:dys-fg} remains robust to the inclusion of~$h$, the performance of \eqref{eq:dys-gf} is more sensitive to the smooth term. Consequently, care must be taken in selecting which formulation to use, especially when the smooth term $h$ plays a significant role in the objective.

%%%%%%%%%%%%%%%%%%%%%%%%%%%%%%%%%%%%%%%%%%%%%%%%%%%%%%%%%%%%%%%%%%%%%%%%%%%%%%%%%%%%%%%%%%%%%%%%%%%%%%%%%%%%%%
\section{Conclusion} \label{sec:conclusion}

This paper presents novel convergence analyses for Douglas--Rachford splitting (DRS) and Davis--Yin splitting (DYS) and their variants obtained by swapping the roles of the two nonsmooth objective functions. For both variants of DRS and one variant of DYS, we establish exact worst-case convergence rates that include the constant factor and establish the tightness through worst-case examples. To the best of our knowledge, this is the first result that establishes the exact worst-case convergence rate for these algorithms including the constant factor. For the other variant of DYS, we establish the best-known convergence rate.

Surprisingly, we show that the swapped DYS algorithm (which we call DYS-$fg$) achieves a faster $\cO (1/(K+1))$ ergodic rate, compared to the standard $\cO(1/K)$ rate of the original DYS. These results are established under a unified primal--dual gap metric and illustrated via concrete examples.  In contrast, for Douglas--Rachford splitting (DRS), both the original and swapped versions exhibit the same convergence rates and nearly identical worst-case instances, an observation that emerges only through tight convergence analyses. This contrast highlights how the presence of a smooth term alters the algorithmic behavior under different update orders.

Our findings were only possible due to the precision of our convergence analysis and suggest that update order is not merely a structural nuance, but one that may affect algorithmic performance. Future work may extend our analysis to broader settings, including a composite extension of \eqref{eq:prob-primal} in which $g$ is replaced by $g \circ A$ for a linear operator $A$, as well as to DYS and DRS for monotone inclusion problems.

%%%%%%%%%%%%%%%%%%%%%%%%%%%%%%%%%%%%%%%%%%%%%%%%%%%%%%%%%%%%%%%%%%%%%%%%%%%%%%%%%%%%%%%%%%%%%%%%%%%%%%%%%%%%%%
%\section*{Acknowledgments}

%%%%%%%%%%%%%%%%%%%%%%%%%%%%%%%%%%%%%%%%%%%%%%%%%%%%%%%%%%%%%%%%%%%%%%%%%%%%%%%%%%%%%%%%%%%%%%%%%%%%%%%%%%%%%%
\bibliography{refs}

\end{document}